\documentclass[12pt,reqno]{amsart}

\usepackage{mathrsfs}

\usepackage{amsmath, amsthm, amssymb}
\usepackage{amsfonts}
\usepackage[ansinew]{inputenc}
\usepackage[dvips]{epsfig}
\usepackage{graphicx}
\usepackage[english]{babel}
\usepackage{hyperref}

\usepackage{tikz}
\usetikzlibrary{through}

\usepackage{cite}
\usepackage{graphicx}
\usepackage{amscd}
\usepackage{bm}
\usepackage{enumerate}

\usepackage{verbatim}
\usepackage{hyperref}
\usepackage{amstext}
\usepackage{latexsym}
\theoremstyle{plain}
\newtheorem{theorem}{Theorem}[section]
\newtheorem{definition}[theorem]{Definition}

\newtheorem{proposition}[theorem]{Proposition}
\newtheorem{lemma}[theorem]{Lemma}
\newtheorem{remark}[theorem]{Remark}

\numberwithin{theorem}{section}
\numberwithin{equation}{section}

\newcommand{\dis}{\displaystyle}

\newcommand{\average}{{\mathchoice {\kern1ex\vcenter{\hrule height.4pt
				width 6pt depth0pt} \kern-9.7pt} {\kern1ex\vcenter{\hrule
				height.4pt width 4.3pt depth0pt} \kern-7pt} {} {} }}

\newcommand{\n }{\nabla }

\newcommand{\be}{\begin{equation}}
\newcommand{\ee}{\end{equation}}

\usepackage{mathrsfs}

\newcommand{\p}{p}

\newcommand{\eps}{\varepsilon}

\renewcommand{\epsilon}{\varepsilon}

\title[]{Bifurcation results for nonlinear eigenvalue problems involving the $(p,q)$-Laplace operator.}
\author{Emmanuel Wend-Benedo Zongo and Bernhard Ruf}
\begin{document}
\address{Universit\'e Paris-Saclay, Institut de Math\'ematiques d'Orsay, France.}
\email{emmanuel.zongo@universite-paris-saclay.fr}

%\author{Bernhard Ruf}
\address{Universit\`a degli studi di Milano, Dipartimento di Matematica ``Federigo Enriques", Italy.}
\email{bernhard.ruf@unimi.it}
\begin{abstract}
    In this paper, we analyze an eigenvalue problem for nonlinear elliptic operators involving homogeneous Dirichlet boundary conditions in a open smooth bounded domain.  We prove bifurcation results from trivial solutions and from infinity for the considered nonlinear eigenvalue problem. We also show the existence of multiple solutions of the nonlinear problem using variational methods.\\
\\
Keywords: nonlinear operators, variational eigenvalues, bifurcation from trivial solution, bifurcation from infinity, multiple solutions.\\
2010 Mathematics Subject Classification: 35J20, 35J92, 35J25.
\end{abstract}
\maketitle
\tableofcontents

\section{Introduction}
 Assume $\Omega\subset\mathbb{R}^N$ ($N\geq 2$) is an open bounded domain with smooth boundary $\partial\Omega.$
In \cite{ZB}, the authors investigated the asymptotic behavior of the spectrum and the existence of multiple solutions of the following nonlinear eigenvalue problem
\begin{equation}\label{article-emma-ruf}
  \begin{cases}
    -\Delta_pu-\Delta u=\lambda u ~~~\text{on}~~\Omega,\\
    u=0~~~~~\text{on}~~\partial\Omega,
  \end{cases}
\end{equation}
where $-\Delta_p$ denotes the $p$-Laplace operator.  In \cite{ZB} it was shown that for $p>2$ there exist eigenvalue branches emanating from $(\lambda_k,0),$ and for $1<p<2$ there exist eigenvalue branches emanating from $(\lambda_k,\infty)$, where $\lambda_k$ stands as the $k$-th Dirichlet eigenvalue of the Laplacian.
In this paper, we consider the following q-homogenous eigenvalue problem with a perturbation by a p-Laplace term:
\begin{equation}\label{e1}
  \begin{cases}
  -\Delta_pu-\Delta_qu=\lambda|u|^{q-2}u~~\text{in}~~\Omega\\
  u=0~~~~\text{on}~~\partial\Omega.
  \end{cases}  
\end{equation}
 The operator $\Delta_s,$ formally defined as $\Delta_su:=\text{div}(|\nabla u|^{s-2}\nabla u)$ for $s=p,q\in(1,\infty)$ is the $s$-Laplacian, $\lambda\in\mathbb{R}$ is a parameter. The $(p,q)$-Laplace operator given by $-\Delta_p-\Delta_q$ appears in a wide range of applications that include biophysics \cite{FP}, plasma physics \cite{SM} and reaction-diffusion equations \cite{AR, CL}. The $(p,q)$-Laplace operator has been widely studied; for some  results related to our studies,  see e.g.,  \cite{BT1,BT2, CD, Ta, MM}. We say that $\lambda$ is a ``first eigenvalue'', if the corresponding eigenfunction $u$ is positive or negative. 
 \par \medskip
 Note that by taking $q=2$ in equation (\ref{e1}), we recover the case of equation (\ref{article-emma-ruf}). We remark however that for $q = 2$ equation \eqref{e1} describes bifurcation (caused by a $p$-Laplace operator) from the {\it linear} equation $-\Delta u - \lambda u$, while for $q \not= 2$  we prove for equation \eqref{e1} the existence of bifurcation branches (again forced by a $p$-Laplace operator) from the eigenvalues of a {\it nonlinear, but $q$-homogenous} equation. Indeed, it was shown in \cite{IPA} that there exists a nondecreasing sequence of {\it variational} positive eigenvalues $\{\lambda_k^D(q)\}_k$ tending to $+\infty$ as $k\rightarrow \infty$ for the following nonlinear and $q$-homogenous eigenvalue problem
\begin{equation}\label{dirichletlap}
\left\{
\begin{array}{l}
-\Delta_q u=\displaystyle \lambda  |u|^{q-2}u~~\text{in $\Omega$},\\
u = \displaystyle 0~~~~~~~~\text{on $\partial\Omega$}.
\end{array}
\right.
\end{equation}
Moreover, it is known that the first eigenvalue of problem (\ref{dirichletlap}) is characterized in the variational point of view by,
$$
\lambda^D_1(q):=\inf_{u\in W^{1,q}_0(\Omega)\backslash\{0\}}\left\{
\frac{\int_{\Omega}|\nabla u|^q~dx}{\int_{\Omega}|u|^q~dx}\right\}
.$$  
%We define the functional $$G(u)=\frac{1}{q}\int_{\Omega}|\nabla u|^q~dx-\frac{\lambda}{q}\int_{\Omega}|u|^q~dx,$$ for $u\in W^{1,q}_0(\Omega).$
We consider the sets $$D_1(q)=\{u\in W^{1,q}_0(\Omega)\backslash\{0\}~:~\int_{\Omega}|u|^qdx=1 \},$$
and $\Sigma,$ the class of closed symmetric (with respect to the origin) subsets of $ W^{1,q}_0(\Omega)\backslash\{0\},$ i.e,
$$\Sigma=\{A\subset W^{1,q}_0(\Omega)\backslash\{0\}:~~A~~\text{closed},~~ A=-A\}.$$ For $A\in\Sigma,$ we define
$$\gamma(A)=\inf\{k\in\mathbb{N}~:~\exists\varphi\in C(A, \mathbb{R}^k\backslash\{0\}),~\varphi(-x)=-\varphi(x)\}.$$ If such $\gamma(A)$ does not exist, we then define $\gamma(A)=+\infty.$ The number $\gamma(A)\in\mathbb{N}\cup\{+\infty\}$ is called the {\it Krasnoselski genus} of $A$.
Let us consider the family of sets 
$$\Sigma_k=\{A\subset\Sigma\cap D_1(q):~\gamma(A)\geq k\}
.$$ 
Following the proof in \cite{IPA}, one shows that one has the following variational characterization of $\lambda^D_k(q)$,
for $k\in\mathbb{N}$,
$$
\lambda^D_k(q)=\displaystyle{\inf_{A\in\Sigma_k}\sup_{u\in A}\int_{\Omega}|\nabla u|^qdx}
.$$
\par \medskip
In this paper, we discuss the nonlinear variational  eigenvalues of equation \eqref{e1}.  
Our main results are:
\par \medskip \noindent
\begin{itemize}
\item[1)] For every fixed $\rho > 0$ there exists a sequence of eigenvalues $\big(\lambda^D_k(p,q;\rho)\big)_k$ with corresponding eigenfunctions $\pm u_k(p,q;\rho)$ satisfying  $\int_\Omega |u_k(p,q;\rho)|^qdx = \rho$, with $\lambda^D_k(p,q;\rho) \to +\infty$ as $k \to \infty$.
\par \medskip \noindent
\item[2)] The variational eigenvalues $\lambda^D_k(q)$ of equation \eqref{dirichletlap} are bifurcation points from $0$ if $p > q$, and bifurcation points from infinity for $1 < p < q$, for the nonlinear eigenvalues $\lambda_k(p,q;\rho)$.
\par \medskip \noindent
\item[3)] For fixed $\lambda \in (\lambda^D_k(q),\lambda_{k+1}^D(q))$ there exist $k$ eigenvalues of \eqref{e1} with \\ $\lambda = \lambda^D_1(p,q;\rho_1)= \dots = \lambda^D_k(p,q;\rho_k)$, with corresponding eigenfunctions $\pm u_k(\p,q;\rho)$ such that $\int_\Omega |u_k|^q = \rho_k$
\end{itemize}
\par \bigskip
 The paper is organized as follows. In section \ref{S1}, we discuss the variational spectrum of the nonlinear problem (\ref{e1}) for $u \in D_\rho$ with fixed $\rho > 0$. In section 3 we give some auxiliary results, and in section 4 we discuss the first eigenvalues of equation \eqref{e1}. Then, in section 5 we discuss the bifurcation phenomena, and finally in section \ref{mult} we prove the multiplicity result.
 \par \medskip
 
 The standard norm of the Lebesgue space $L^s(\Omega)$ and the Sobolev space $W^{1,s}_0(\Omega)$ will be denoted respectively by $\|\cdot \|_s = (\int_{\Omega}|\cdot|^sdx)^{1/s}$ and $\|\cdot \|_{1,s}=(\int_{\Omega}|\nabla(\cdot)|^sdx)^{1/s}$. We also denote by $\langle\, ,\rangle,$ the duality product between $W^{1,s}_0(\Omega)$ and its dual.
 \par \medskip
\section{The spectrum of problem (\ref{e1})}\label{S1}
In this section we show that equation \eqref{e1} has for every given $\rho > 0$ a sequence of eigenvalues $\lambda_k^D(p,q,\rho)$, with associated eigenfunctions $u_k(p,q,\rho)$ and $\int_\Omega |u_k(p,q,\rho)|^qdx = \rho$.
\begin{definition}
We say that $u\in W^{1,p}_0(\Omega)$ $($if $p>q$ $)$ or $u\in W^{1,q}_0(\Omega)$ $($if $p<q$ $)$  is a weak solution of problem (\ref{e1}) if the following integral equality holds:
\begin{equation}\label{e2}
    \int_{\Omega}|\nabla u|^{p-2}\nabla u\cdot\nabla v~dx+\int_{\Omega}|\nabla u|^{q-2}\nabla u\cdot\nabla v~dx=\lambda\int_{\Omega}|u|^{q-2}u\,v~dx,
\end{equation}
for all $v\in W^{1,p}_0(\Omega)\cap  W^{1,q}_0(\Omega)$. 
\end{definition}
\par \medskip\noindent
We say that $\lambda\in\mathbb{R}$ is an eigenvalue of problem (\ref{e1}) if there exists an eigenfunction $u_{\lambda}\in (W^{1,p}_0(\Omega)\cap  W^{1,q}_0(\Omega))\backslash\{0\}$ associated to $\lambda$ such that relation (\ref{e2}) holds.
\par \bigskip
We say that $\lambda^D_1(p,q,\rho)$ is a first eigenvalue of equation \eqref{e1} if the corresponding eigenfunction $u_1(p,q;\rho)$ is a minimizer of the following expression, for some $\rho > 0$,
\begin{equation}\label{e3}
 c_1(p,q;\rho):=\inf_{\{u\in  W^{1,p}_0(\Omega)\cap  W^{1,q}_0(\Omega), \int_\Omega |u|^q = \rho\}} \Big(\frac{1}{p}\int_{\Omega}|\nabla u|^pdx+\frac{1}{q}\int_{\Omega}|\nabla u|^qdx \Big). 
\end{equation}
Note that $\lambda^D_1(p,q;\rho)$ satisfies
$$
\int_{\Omega}|\nabla u_1|^pdx+\int_{\Omega}|\nabla u_1|^qdx
= \lambda^D_1(p,q;\rho) \int_\Omega |u_1|^qdx = \lambda^D_1(p,q;\rho) \rho.
$$
\par \smallskip
\begin{proposition}
If it holds $\lambda\leq \lambda^D_1(q)$ then problem (\ref{e2}) has no nontrivial solutions.
\end{proposition}
\begin{proof}
Suppose by contradiction that there exists $\lambda<\lambda^D_1(q)$ which is an eigenvalue of problem (\ref{e1}) with $u_{\lambda}\in (W^{1,p}_0(\Omega)\cap  W^{1,q}_0(\Omega))\backslash\{0\}$ the corresponding eigenfunction. Let $v=u_{\lambda}$ in relation (\ref{e2}), we then have $$\int_{\Omega}|\nabla u_{\lambda}|^pdx+\int_{\Omega}|\nabla u_{\lambda}|^qdx=\lambda\int_{\Omega}|u_{\lambda}|^qdx.$$ On the other hand, we have
\begin{equation}\label{p3}
   \lambda^D_1(q) \int_{\Omega}|u_{\lambda}|^qdx\leq \int_{\Omega}|\nabla u_{\lambda}|^qdx,
\end{equation}
and subtracting by $\lambda \displaystyle{\int_{\Omega}}|u_{\lambda}|^qdx$ from both sides of (\ref{p3}), it follows that 
$$(\lambda^D_1(q)-\lambda)\int_{\Omega}|u_{\lambda}|^qdx\leq \int_{\Omega}|\nabla u_{\lambda}|^qdx-\lambda\int_{\Omega}|u_{\lambda}|^qdx.$$ This implies that $$0<(\lambda^D_1(q)-\lambda)\int_{\Omega}|u_{\lambda}|^qdx\leq \int_{\Omega}|\nabla u_{\lambda}|^qdx+\int_{\Omega}|\nabla u_{\lambda}|^pdx-\lambda\int_{\Omega}|u_{\lambda}|^qdx=0.$$ Hence $\lambda<\lambda^D_1(q)$ is not an eigenvalue of problem (\ref{e1}) with $u_{\lambda}\neq 0.$
\par \medskip
Now, assume that $\lambda=\lambda^D_1(q)$ is an eigenvalue of equation (\ref{e1}), thus there exists an eigenfunction $u_{\lambda^D_1(q)}\in (W^{1,p}_0(\Omega)\cap  W^{1,q}_0(\Omega))\backslash\{0\} $ associated to $\lambda^D_1(q)$ such that relation (\ref{e2}) holds. Letting $v=u_{\lambda^D_1(q)}$ in (\ref{e2}), we obtain 
$$\int_{\Omega}|\nabla u_{\lambda^D_1(q)}|^pdx+\int_{\Omega}|\nabla u_{\lambda^D_1(q)}|^qdx=\lambda^D_1(q)\int_{\Omega}|u_{\lambda^D_1(q)}|^qdx.$$ Since $$\lambda^D_1(q)\int_{\Omega}|u_{\lambda^D_1(q)}|^qdx\leq \int_{\Omega}|\nabla u_{\lambda^D_1(q)}|^qdx,$$ it follows that $$\int_{\Omega}|\nabla u_{\lambda^D_1(q)}|^pdx+\int_{\Omega}|\nabla u_{\lambda^D_1(q)}|^qdx\leq \int_{\Omega}|\nabla u_{\lambda^D_1(q)}|^qdx$$ and then $u_{\lambda^D_1(q)}=0$ by the Poincar\'e inequality. This concludes the proof.
\end{proof}
%\begin{remark}
    %Let $\lambda^D_1(p,q)$ be as in (\ref{e3}),
    %=\inf_{u\in  (W^{1,p}_0(\Omega)\cap  W^{1,q}_0(\Omega))\backslash\{0\}}\left\{\frac{\frac{1}{p}\int_{\Omega}|\nabla u|^pdx+\frac{1}{q}\int_{\Omega}|\nabla u|^qdx}{\frac{1}{q}\int_{\Omega}|u|^qdx}\right\}.$ 
    %then $\lambda^D_1(p,q)=\lambda^D_1(q)$ for any $p>q$ or $p<q.$
%\end{remark}
\begin{proposition}
The first eigenfunctions $u_1^{\lambda}$ associated to some $\lambda\in (\lambda^D_1(q),\infty)$ are positive or negative in $\Omega.$
\begin{proof}
Let $u^{\lambda}_1\in (W^{1,p}_0(\Omega)\cap W^{1,1}_0(\Omega))\setminus\{0\}$ be a first eigenfunction associated to $\lambda\in(\lambda^D_1(q),\infty), $ then
$$\int_{\Omega}|\n u^{\lambda}_1|^pdx+\int_{\Omega}|\n u^{\lambda}_1|^qdx=\lambda\int_{\Omega}| u^{\lambda}_1|^qdx,$$ which means that $u^{\lambda}_1$ achieves the infimum in the definition of $\mu_1(p,q;\rho)$, with $\rho = \frac 1q \int_\Omega |u|^q$. On the other hand, we have $\big\|\n|u^{\lambda}_1|\big\|_{1,s}=\|\n u^{\lambda}_1\|_{1,s}$ for $s=p,q$ and $\big\||u^{\lambda}_1|\big\|_q=\|u^{\lambda}_1\|_q$ since $\big|\n|u^{\lambda}_1|\big|=|\n u^{\lambda}_1| $ and $\big||u^{\lambda}_1|\big|=| u^{\lambda}_1|$ almost everywhere. Then, it follows that $|u^{\lambda}_1|$ achieves also the infimum in the definition of $\mu_1(p,q;\rho).$ Therefore by the Harnack inequality, we have $|u^{\lambda}_1|>0$ for all $x\in\Omega$ and consequently $u^{\lambda}_1$ is either positive or negative in $\Omega.$
\end{proof}
\end{proposition}
The Palais-Smale condition plays an important role in the minimax argument, and we recall here its definition.
\begin{definition}
A $C^1$ functional $I$ defined on a smooth submanifold $M$ of a Banach space $X$ is said to satisfy the Palais-Smale condition on $M$ if any sequence $\{u_n\}\subset M$ satisfying that $\{I(u_n)\}_n$ is bounded and $\big(I\big|_M\big)'(u_n)\rightarrow 0$ as $n\rightarrow +\infty$ has a convergent subsequence.
\end{definition}
Next, we start the discussion about the existence of eigenvalues for problem (\ref{e1}). We note that these eigenvalues depend on $\rho(u)=\int_{\Omega}|u|^qdx.$ The proofs of the following two theorems rely on \cite[Proposition 10.8]{AM}.
\begin{theorem}\label{sequence1}
Let $p>q.$ Then, for a given $\rho>0,$ there exists a nondecreasing sequence of critical values $c_k(p,q;\rho)$ with associated nonlinear eigenvalues $\lambda^D_k(p,q;\rho)\to +\infty,$ as $k\to +\infty$ and with corresponding eigenfunctions $u_k(p,q;\rho)\in W^{1,p}_0(\Omega)$ for problem (\ref{e1}).
\end{theorem}
\begin{proof}
Let $D_{\rho}(p,q)=\{u\in W^{1,p}_0(\Omega)\, : \, \int_{\Omega}|u|^qdx=\rho\}$, and $\Sigma_k(p,q)=\{A \subset D_\rho(p,q), A \in \Sigma \hbox{ and} \ \gamma(A) \geq k\},$ where $\Sigma=\{A\subset W^{1,p}_0(\Omega): \, A\,\text{closed}, A=-A\}.$ Set 
\begin{equation}\label{ck}
c_k(p,q;\rho)=\inf_{A\in\Sigma_k(p,q)}\sup_{u\in A}\left(\frac{1}{p}\int_{\Omega}|\nabla u|^pdx+\frac 1q\int_{\Omega}|\nabla u|^qdx\right)>0
\end{equation}
Let us show that $I(u)=\displaystyle{\frac{1}{p}\int_{\Omega}|\nabla u|^pdx+\frac 1q\int_{\Omega}|\nabla u|^qdx}$ satisfies the Palais-Smale (PS) condition on $D_{\rho}(p,q)$. Let $\{u_n\}\subset D_{\rho}(p,q)$ a (PS) sequence, i.e, for all $n,$ $K>0$ $|I(u_n)|\leq K$ and $(I\big|_{D_\rho})'(u_n)\rightarrow 0$ in $W^{-1,p'}(\Omega)$ as $n \rightarrow \infty.$ We first show that $\{u_n\} \subset D_{\rho}(p,q)$ is bounded in $W^{1,p}_0(\Omega).$ Since $u_n\in W^{1,q}_0(\Omega),$ with the Poincar\'e inequality, we have $\int_{\Omega}|u_n|^qdx\leq K\int_{\Omega}|\nabla u_n|^qdx $ and it follows that 
$$
K\geq |I(u_n)|\geq \frac{q}{p}\int_{\Omega}|\nabla u_n|^pdx+\frac{1}{C}\int_{\Omega}|u_n|^qdx=\frac{q}{p}\|u_n\|^p_{1,p}+\frac{\rho}{C}.
$$ 
Then $\{u_n\} \subset D_{\rho}(p,q)$ is bounded in $ W^{1,p}_0(\Omega).$ We can assume that up to a subsequence, still denoted $\{u_n\}$, there exists $u\in W^{1,p}_0(\Omega)$ such that $u_n \rightharpoonup u$ in $W^{1,p}_0(\Omega)$ and $u_n \to u$ in $L^q$. Now, we show that $u_n$ converges strongly to $u$ in $W^{1,p}_0(\Omega).$ Since 
$(I\big|_{D_\rho})'(u_n)\rightarrow 0$ in $W^{-1,p'}(\Omega)$ as $n\rightarrow \infty,$ there exists $\mu_n \in \mathbb R$ and $\varepsilon_n \to 0$ in $W^{-1,p'}_0(\Omega)$ such that $I'(u_n)v - \mu_n \int_\Omega |u_n|^{q-2}u_n v = \langle \varepsilon_n,v\rangle$. We have $I'(u_n)u_n - \mu_n\int_\Omega |u_n|^q \to 0$, and since $I'(u_n)u_n \le c I(u_n)\le c$  it follows that $|\mu_n| \le c$. From this we obtain that $I'(u_n)(u_n-u)\rightarrow 0$ and $I'(u)(u_n-u)\rightarrow 0$ as $n\rightarrow\infty.$ Therefore, 
\begin{align*}
  o(1) =   \langle I'(u_n)-I'(u), u_n-u\rangle&=\int_{\Omega}(|\nabla u_n|^{p-2}\nabla u_n-|\nabla u|^{p-2}\nabla u)\cdot \nabla (u_n-u)dx\\ &+\underbrace{\int_{\Omega}(|\nabla u_n|^{q-2}\nabla u_n-|\nabla u|^{q-2}\nabla u)\cdot \nabla (u_n-u)dx}_{:=Q}.
\end{align*}
Using Lemma \ref{inevec} below and the fact that the underbraced quantity $Q$ is positive (see Remark \ref{r1}), it follows that 
$$\langle I'(u_n)-I'(u), u_n-u\rangle\geq c_2\|u_n-u\|^p_{1,p}.$$ This shows that $u_n$ converges strongly to $u$ in $W^{1,p}_0(\Omega)$ as $n\rightarrow\infty $ since $\langle I'(u_n)-I'(u), u_n-u\rangle\rightarrow 0$ as $n\rightarrow\infty.$\\
\\
In order to end the proof, let us show that if $c=c_k(p,q)=\dots=c_{k+m-1}(p,q),$ then the set $K_c$ of critical points of $I$ at the critical level $c$ has a genus $\gamma(K_c)\geq m.$
We consider the level set at $c$, $$K_c:=\{u\in D_{\rho}(p,q):~I(u)=c~,~I'(u)=0\}.
$$ 
We have that $K_c$ is compact since the functional $I$ satisfies the Palais-Smale condition and $0\notin K_c$ since $c>0=I(0)$. In addition, we have $I(u)=I(-u)$. Hence $K_c\in \Sigma$. Assume by contradiction that $\gamma(K_c) \le  m-1$. Take $A_\varepsilon \in \Sigma_{k+m-1}$ such that $\sup_{A_\varepsilon} I(u) \le c+\varepsilon$. By the properties of the genus, there exists a $\delta$-neighborhood $N_\delta$ of $K_c$ such that $\gamma(N_\delta) = \gamma(K_c)$, and
$\gamma(A_\varepsilon \setminus N_\delta) \ge \gamma(A_\varepsilon) - \gamma(N_\delta) \ge k+m-1 - (m-1) = k$. By the deformation theorem there exists a homeomorphism $\eta(1,\cdot)$ such that $I(u) \le c-\varepsilon$, for $u \in \eta(1,A_\varepsilon \setminus N_\delta)$. Then we arrive at the contradiction
$$
c = \inf_{A\in \Sigma_k}\sup_{u \in A} I(u)\le \sup_{\eta(1,A_\varepsilon \setminus N_\delta)}I(u) \le c-\varepsilon
$$
Hence, $\gamma(K_c) \ge m$.
\par
With a compactness argument one shows that $c_k(p,q;\rho) \to \infty$ as $k \to \infty$. 
\par \smallskip
For the corresponding eigenvalues $\lambda^D_k(p,q,\rho)$ we then have
$$ 
\int_\Omega |\nabla u_k|^p dx + \int_\Omega |\nabla u_k|^q dx
= \lambda^D_k(p,q;\rho) \int_\Omega |u_k|^qdx = \lambda^D_k(p,q;\rho) \, \rho
$$
Thus $\lambda^D_k(p,q;\rho) \, \rho > c_k(p,q;\rho)$, for all $k$ (and fixed $\rho$), and hence also $\lambda^D_k(p,q;\rho) \to \infty$ as $k \to \infty$.
\end{proof}
For $p<q$ one has the analogous result:
\begin{theorem}\label{sequence2}
Let $p<q$ be given.Then, for a given $\rho>0,$ there exists a nondecreasing sequence of critical values $c_k(p,q;\rho)$ with associated nonlinear eigenvalues $\lambda^D_k(p,q;\rho)\to +\infty,$ as $k\to +\infty$ and with corresponding eigenfunctions $u_k(p,q;\rho)\in W^{1,q}_0(\Omega)$ for problem (\ref{e1}).
\end{theorem}
\begin{proof}
Let $D_{\rho}(p,q)=\{u\in W^{1,q}_0(\Omega)~:~\int_{\Omega}|u|^qdx=\rho\}$, and $\Sigma_k(p,q)=\{A\subset\Sigma~:~\gamma(A\cap D_{\rho}(p,q))\geq k\},$ where $\Sigma=\{A\subset W^{1,q}_0(\Omega):~~A~~\text{closed},~~ A=-A\}.$ Set 
$$
b_k(p,q)=\inf_{A\in\Sigma_k(p,q)}\sup_{u\in A}\left(\frac{1}{p}\int_{\Omega}|\nabla u|^pdx+\frac 1q\int_{\Omega}|\nabla u|^qdx\right)>0.
$$ 
Similar to the proof of Theorem \ref{sequence1}, one shows that:
\begin{enumerate}
    \item the functional $I(u)=\displaystyle{\frac{1}{p}\int_{\Omega}|\nabla u|^pdx+ \frac 1q \int_{\Omega}|\nabla u|^qdx}$ satisfies the (PS) condition on $D_{\rho}(p,q)$, and
    \item if $b=b_k(p,q)=\dots=b_{k+m-1}(p,q),$ then the set $K_b$ of critical points of $I$ at the critical level $b$ has a genus $\gamma(K_b)\geq m.$
\end{enumerate}
%First, let us show that $I(u)=\int_{\Omega}|\nabla u|^pdx+\int_{\Omega}|\nabla u|^qdx$ satisfies the Palais-Smale (PS) condition on $D_{\rho}(p,q).$
\end{proof}
We note that the results of Theorem \ref{sequence1}-\ref{sequence2} are illustrated in 
figure \ref{fig1} in section \ref{mult}.
\section{Auxiliary results}
\begin{remark}\label{r1}
\textup{Let $p>q$. We recall that the nonlinear operator $\Theta: W^{1,p}_0(\Omega)\rightarrow W^{-1,q'}(\Omega)\subset W^{-1,p'}(\Omega)$ defined by $$\langle\Theta u, v\rangle= \int_{\Omega}|\nabla u|^{p-2}\nabla u\cdot\nabla v~dx+\int_{\Omega}|\nabla u|^{q-2}\nabla u\cdot\nabla v~dx$$is continuous and so it is demi-continuous. The operator $\Theta $ is said to be demi-continuous if $\Theta$ satisfies that whenever $u_n\in W^{1,p}_0(\Omega)$ converges to some $u\in W^{1,p}_0(\Omega)$ then $\Theta u_n\rightharpoonup\Theta u$ as $n\rightarrow\infty.$
\\
In addition, we claim that the operator $\Theta$ satisfies the following condition: for any $u_n\in W^{1,p}_0(\Omega)$ satisfying $u_n\rightharpoonup u$ in $W^{1,p}_0(\Omega)$ and $\limsup\limits_{n\rightarrow\infty}\langle\Theta u_n, u_n-u\rangle\leq 0,$ then $u_n\rightarrow u$ in $W^{1,p}_0(\Omega)$ as $n\rightarrow\infty.$ The same result hold in the case where $p<q.$}
\end{remark}
\indent Indeed, assume that $u_n\rightharpoonup u$ in $W^{1,p}_0(\Omega)$ and $\limsup\limits_{n\rightarrow\infty}\langle\Theta u_n, u_n-u\rangle\leq 0.$ Hence $u_n$ converges strongly to $u$ in $L^p(\Omega)$ and one has
%\vspace{-0.8cm}
\begin{equation*}
\begin{split}
  0&\geq \limsup\limits_{n\rightarrow\infty}\langle\Theta u_n-\Theta u, u_n-u\rangle\\ &=  
  \limsup\limits_{n\rightarrow\infty}\int_{\Omega}\left[|\nabla u_n|^{p-2}\nabla u_n-|\nabla u|^{p-2}\nabla u+|\nabla u_n|^{q-2}\nabla u_n-|\nabla u|^{q-2}\nabla u)\right]\cdot\nabla(u_n-u)dx.
\end{split}
\end{equation*}
On the other hand, for any $\nabla u_n, \nabla u\in (L^p(\Omega))^N$, one has,
$$\begin{array}{ll}
 \dis\int_{\Omega}\hspace{-0,2cm}&\dis(|\nabla u_n|^{p-2}\nabla u_n -|\nabla u|^{p-2}\nabla u)\cdot\nabla(u_n-u)dx \dis=\int_{\Omega}(|\nabla u_n|^p+|\nabla u|^p-|\nabla u_n|^{p-2}\nabla u_n\cdot\nabla u
 \vspace{0.2cm}\\ 
 & \dis-|\nabla u|^{p-2}\nabla u\cdot\nabla u_n)dx 
 \vspace{0.2cm}\\
 &\dis\geq \int_{\Omega}(|\nabla u_n|^p+|\nabla u|^p)dx-\Big(\int_{\Omega}|\nabla u_n|^pdx\Big)^{1/p'}\times \Big(\int_{\Omega}|\nabla u|^p dx\Big)^{1/p}- \vspace{0.2cm}\\
 &\dis-\Big(\int_{\Omega}|\nabla u_n|^pdx\Big)^{1/p}
\times \Big(\int_{\Omega}|\nabla u|^pdx\Big)^{1/p'}
 \vspace{0.2cm}\\ 
 &\dis=\Big[\Big(\int_{\Omega}|\nabla u_n|^pdx\Big)^{\frac{p-1}{p}}-\Big(\int_{\Omega}|\nabla u|^pdx\Big)^{\frac{p-1}{p}}\Big]\times \Big[\Big(\int_{\Omega}|\nabla u_n|^pdx\Big)^{\frac{1}{p}}-\Big(\int_{\Omega}|\nabla u|^pdx\Big)^{\frac{1}{p}}\Big]
 \vspace{0.2cm}\\ 
 &\dis=\big(\|u_n\|^{p-1}_{1,p}-\|u\|^{p-1}_{1,p}\big)\big(\|u_n\|^{p}_{1,p}-\|u\|^{p}_{1,p}\big)\geq 0.
    \end{array}
$$
We then deduce from this inequality that $\int_{\Omega}|\nabla u_n|^pdx\rightarrow \int_{\Omega}|\nabla u|^pdx$ as $n\rightarrow\infty$ and similarly $\int_{\Omega}|\nabla u_n|^qdx\rightarrow \int_{\Omega}|\nabla u|^qdx$ as $n\rightarrow\infty.$ Consequently $u_n$ converges strongly to $u$ in $W^{1,p}_0(\Omega)\subset W^{1,q}_0(\Omega).$ 
\begin{proposition}
Assume that $p>q.$ If $(\lambda,0)$ is a bifurcation point of solutions of problem (\ref{e1}) then $\lambda$ is an eigenvalue of problem (\ref{dirichletlap}).
\end{proposition}
\begin{proof}
 Since $(\lambda,0)$ is a bifurcation point from zero of solutions of problem (\ref{e1}), there is a sequence of nontrivial solutions of problem (\ref{e1}) such that $\lambda_n\rightarrow\lambda$ and $\|u_n\|_{1,p}\rightarrow 0$ in $W^{1,p}_0(\Omega).$ We then have 
\begin{equation}\label{p1}
   \int_{\Omega}|\nabla u_n|^{p-2}\nabla u_n\cdot\nabla v~dx+\int_{\Omega}|\nabla u_n|^{q-2}\nabla u_n\cdot\nabla v~dx=\lambda_n\int_{\Omega}|u_n|^{q-2}u_nv~dx,
\end{equation}

Let $w_n=u_n/\|u_n\|_{1,p}.$ Plugging this change of variable into equation (\ref{p1}), we get
\begin{equation}\label{p2}
     \|u_n\|^{p-q}_{1,p}\int_{\Omega}|\nabla w_n|^{p-2}\nabla w_n\cdot\nabla v~dx+\int_{\Omega}|\nabla w_n|^{q-2}\nabla w_n\cdot\nabla v~dx=\lambda_n\int_{\Omega}|w_n|^{q-2}w_nv~dx
\end{equation}
With Remark \ref{r1}, it follows that $$ \|u_n\|^{p-q}_{1,p}\int_{\Omega}|\nabla w_n|^{p-2}\nabla w_n\cdot\nabla v~dx+\int_{\Omega}|\nabla w_n|^{q-2}\nabla w_n\cdot\nabla v~dx\rightarrow \int_{\Omega}|\nabla w|^{q-2}\nabla w\cdot\nabla v~dx$$ as $n\rightarrow\infty$ since $\|u_n\|_{1,p}\rightarrow 0$ by assumption and $\lambda_n\int_{\Omega}|u_n|^{q-2}u_nv~dx$ converges to $\lambda\int_{\Omega}|u|^{q-2}uv~dx$ as $n\rightarrow\infty$. Thus, we obtain that $$\int_{\Omega}|\nabla w|^{q-2}\nabla w\cdot\nabla v~dx=\lambda\int_{\Omega}|u|^{q-2}uv~dx$$ for all $v\in W^{1,p}_0(\Omega).$ 
\end{proof}
The following lemma will be used in some occasions.
\begin{lemma}[\cite{PL}]\label{inevec}
There exist constants $c_1,c_2$ such that for all $x_1,x_2\in\mathbb{R}^N,$ we have the following vector inequalities for $1<s<2$
 \begin{equation*}
 (|x_2|^{s-2}x_2-|x_1|^{s-2}x_1)\cdot(x_2-x_1)\geq c_1(|x_2|+|x_1|)^{s-2}|x_2-x_1|^2,
 \end{equation*}
 and for $s>2$
 \begin{equation*}
 (|x_2|^{s-2}x_2-|x_1|^{s-2}x_1)\cdot(x_2-x_1)\geq c_2|x_2-x_1|^s.
 \end{equation*}
\end{lemma}
\par \medskip

\section{First eigenvalues}

In this section we prove that every $\lambda > \lambda_1^D(q)$ is a first eigenvalue of problem \eqref{e1}.
\par \bigskip

We define the  energy functional $E_{\lambda}: W^{1,p}_0(\Omega)\cap  W^{1,q}_0(\Omega)\rightarrow \mathbb{R}$ associated to relation (\ref{e2}) by
\begin{equation}\label{functional}
    E_{\lambda}(u)=\frac{1}{p}\int_{\Omega}|\nabla u|^pdx+\frac{1}{q}\int_{\Omega}|\nabla u|^qdx-\frac{\lambda}{q}\int_{\Omega}|u|^qdx.
\end{equation}

\begin{lemma}\label{coercivite}
Suppose that $p>q.$ Then for each $\lambda>0,$ the functional $E_{\lambda}$ defined in (\ref{functional}) is coercive.
\end{lemma}
\begin{proof}
If $p>q,$ We have that $W^{1,p}_0(\Omega)\subset W^{1,q}_0(\Omega)$ and the following inequalities hold true
\begin{enumerate}
    \item $ \frac{1}{p}\displaystyle{\int_{\Omega}}|\nabla u|^pdx+\frac{1}{q}\displaystyle{\int_{\Omega}}|\nabla u|^qdx\geq \frac{1}{p}\displaystyle{\int_{\Omega}}|\nabla u|^pdx,$
    \item $\displaystyle{\int_{\Omega}}|\nabla u|^qdx\leq C\|u\|^q_{1,p}$ (using the H\"older inequality).
\end{enumerate}
With items (i) and (ii) we obtain $E_{\lambda}(u)\geq \frac{1}{p}\|u\|^p_{1,p}-\tilde{C}\|u\|^q_{1,p}$ and consequently $E_{\lambda}(u)\rightarrow +\infty$ as $\|u\|_{1,p}\rightarrow +\infty.$
%If $p<q,$ we clearly see that $E_{\lambda}$ is not coercive.
\end{proof}
\begin{remark}
We notice that $E_{\lambda}$ is not bounded below if $p<q$ and $\lambda>\lambda^D_1(q)$ since for every $u=u_1,$ the first eigenfunction of (\ref{dirichletlap}) with $\int_{\Omega}|u_1|^qdx=1,$ we have $$E_{\lambda}(tu)=\frac{t^p}{p}\|u_1\|^p_{1,p}+\frac{t^q}{q}(\lambda^D_1(q)-\lambda)\rightarrow -\infty$$ as $t\rightarrow\infty.$ 
\end{remark}
\begin{theorem}
Let $p>q$. Then every $\lambda\in  (\lambda^D_1(q),\infty)$ is a first eigenvalue of problem (\ref{e1}).
\end{theorem}
\begin{proof}
Standard arguments show that $E_{\lambda}\in C^1(W^{1,p}_0(\Omega),\mathbb{R})$ with its derivative given by $$\langle E'_{\lambda}(u),v\rangle=\int_{\Omega}|\nabla u|^{p-2}\nabla u\cdot \nabla v ~dx+\int_{\Omega}|\nabla u|^{q-2}\nabla u\cdot \nabla v ~dx-\lambda\int_{\Omega}|u|^{q-2}u~v~dx,$$ for all $v\in W^{1,p}_0(\Omega)\subset W^{1,q}_0(\Omega).$ On the other hand $E_{\lambda}$ is weakly lower semi-continuous on $W^{1,p}_0(\Omega)\subset W^{1,q}_0(\Omega)$ since $E_{\lambda}$ is a continuous convex functional. This fact and Lemma \ref{coercivite} allow one to apply a direct calculus of variations result in order to obtain the existence of global minimum point of $E_{\lambda}$. We denote by $u_0$ such a global minimum point, i.e, $E_{\lambda}(u_0)=\min\limits_{u\in W^{1,p}_0(\Omega) }E_{\lambda}(u).$ We observe that for $u_0=sw_1$ (where $w_1$ stands for the corresponding eigenfunction of $\lambda_1^D(q)$), we have 
$$E_{\lambda}(u_0)=\frac{s^p}{p}\int_{\Omega}|\nabla w_1|^p~dx+\frac{s^q}{q}(\lambda_1^D(q)-\lambda)<0$$ for $s$ small enough. So there exists $u_{\lambda}\in W^{1,p}_0(\Omega)$ such that $E_{\lambda}(u_{\lambda})<0.$ But $E_{\lambda}(u_0)\leq E_{\lambda}(u_{\lambda})<0,$ which implies that $u_0\in W^{1,p}_0(\Omega)\backslash\{0\}.$ We also have that $\langle E'_{\lambda}(u_0),v\rangle=0, \forall \ v \in E,$ and this concludes the proof.
\end{proof}
\par \bigskip
To treat the case where $p<q,$ we constrain $E_{\lambda}$ on the Nehari set
\begin{equation*}
\begin{split}
    \mathcal{N}_{\lambda}&=\{u\in W^{1,q}_0(\Omega)/~u\neq 0,~\langle E'_{\lambda}(u),u\rangle=0\}\\
    &=\{u\in W^{1,q}_0(\Omega)/~u\neq 0,~\int_{\Omega}|\nabla u|^pdx+\int_{\Omega}|\nabla u|^qdx=\lambda\int_{\Omega}|u|^qdx\}.
\end{split}
\end{equation*}
On $\mathcal{N}_{\lambda},$ the functional $E_{\lambda}$ reads as $E_{\lambda}(u)=(\frac{1}{p}-\frac{1}{q})\displaystyle{\int_{\Omega}}|\nabla u|^pdx>0.$\\
\\
This shows at once that $E_{\lambda}$ is coercive in the sense that if $u\in\mathcal{N}_{\lambda}$ satisfies $\|u\|_{1,p}\rightarrow\infty,$ then $E_{\lambda}(u)\rightarrow \infty.$\\
\\
We define $m=\inf\limits_{u\in\mathcal{N}_{\lambda}}E_{\lambda}(u),$ and we show through a series of propositions that $m$ is attained by some $u\in\mathcal{N}_{\lambda}$ which is a critical point of $E_{\lambda}$ considered on the whole space $W^{1,q}_0(\Omega)\subset W^{1,p}_0(\Omega)$ and therefore a solution to equation (\ref{e1}).
\begin{proposition}\label{nonvide}
The set $\mathcal{N}_{\lambda}$ is not empty for $\lambda>\lambda^D_1(q).$
\end{proposition}
\begin{proof}
Since $\lambda>\lambda^D_1(q)$ there exists $u\in W^{1,q}_0(\Omega)$ not identically zero such that $\int_{\Omega}|\nabla u|^qdx<\lambda\int_{\Omega}|u|^qdx.$ We then see that $tu\in \mathcal{N}_{\lambda}$ for some $t>0.$ Indeed, $tu\in \mathcal{N}_{\lambda}$ is equivalent to
$$t^p\int_{\Omega}|\nabla u|^pdx+t^q\int_{\Omega}|\nabla u|^qdx=t^q\lambda\int_{\Omega}|u|^qdx,$$ which is solved by $t=\left(\frac{\int_{\Omega}|\nabla u|^pdx}{\lambda\int_{\Omega}|u|^qdx-\int_{\Omega}|\nabla u|^qdx}\right)^{\frac{1}{q-p}}>0.$
\end{proof}
\begin{proposition}\label{minborne}
Every minimizing sequence for $E_{\lambda}$ on $\mathcal{N}_{\lambda}$ is bounded in $W^{1,q}_0(\Omega).$
\end{proposition}
\begin{proof}
Let $\{u_n\}_{n\geq0}\subset \mathcal{N}_{\lambda}$ be a minimizing sequence of $E_{\lambda}|_{\mathcal{N}_{\lambda}}$, i.e. $E_{\lambda}(u_n)\rightarrow m=\displaystyle{\inf_{v\in\mathcal{N}_{\lambda}}}E_{\lambda}(v).$
Then
\begin{equation}\label{m1}
\lambda\int_{\Omega}|u_n|^q~dx-
\int_{\Omega}|\nabla u_n|^q~dx=\int_{\Omega}|\nabla u_n|^p~dx\rightarrow\left(\frac{1}{p}-\frac{1}{q}\right)^{-1}m,~\text{as $n\rightarrow\infty$}.
\end{equation}
Suppose on the contrary that $\{u_n\}_{n\geq0}$ is not bounded i.e. $\displaystyle\int_{\Omega}|\nabla u_n|^q~dx\rightarrow\infty$ as $n\rightarrow\infty$. Then we have $\displaystyle{\int_{\Omega}}|u_n|^q~dx
\rightarrow\infty$ as $n\rightarrow\infty$, using relation (\ref{m1}). We set $w_n=\frac{u_n}{\|u_n\|_q}.$ Since $\displaystyle{\int_{\Omega}}|\nabla u_n|^q~dx<\lambda\displaystyle{\int_{\Omega}}|u_n|^q~dx
$, we deduce that $\displaystyle{\int_{\Omega}}|\nabla w_n|^q~dx<\lambda,$ for each $n$ and  $\|w_n\|_{1,q}<\lambda^{1/q}.$ Hence $\{w_n\}\subset W^{1,q}_0(\Omega)$ is bounded in $W^{1,q}_0(\Omega).$ Therefore there exists $w_0\in W^{1,q}_0(\Omega)$ such that $w_n\rightharpoonup w_0$ in $W^{1,q}_0(\Omega)\subset W^{1,p}_0(\Omega)$ and $w_n\rightarrow w_0$ in $L^q(\Omega).$ 
Dividing relation (\ref{m1}) by $\|u_n\|^p_{q}$, we get $$\int_{\Omega}|\nabla w_n|^p~dx=\frac{\lambda\displaystyle{\int_{\Omega}}|u_n|^q~dx-
\displaystyle{\int_{\Omega}}|\nabla u_n|^q~dx}{\|u_n\|^p_{q}}\rightarrow 0~~\text{as~$n \rightarrow\infty$},$$ since $\lambda\int_{\Omega}|u_n|^q~dx-
\int_{\Omega}|\nabla u_n|^q~dx\rightarrow\left(\frac{1}{p}-\frac{1}{q}\right)^{-1}m<\infty~\text{as $n\rightarrow\infty$}$ and $\|u_n\|^p_{q}\rightarrow\infty$ as $n\rightarrow\infty$. On the other hand, since $w_n\rightharpoonup w_0$ in $W^{1,p}_0(\Omega)$, we infer that $\displaystyle{\int_{\Omega}}|\nabla w_0|^p~dx\leq \lim_{n\rightarrow\infty}\inf\displaystyle
{\int_{\Omega}}|\nabla w_n|^p~dx=0$ and consequently $w_0=0$. Therefore $w_n\rightarrow 0$ in $L^q(\Omega),$ which is a contradiction since $\|w_n\|_{q}=1$. Hence, $\{u_n\}_{n\geq 0}$ is bounded in $W^{1,q}_0(\Omega).$
\end{proof}
\begin{proposition}\label{minpositive}
We have $m=\inf\limits_{u\in\mathcal{N}_{\lambda}}E_{\lambda}(u)>0.$
\end{proposition}
\begin{proof}
Assume by contradiction that $m=0$. Then, for $\{u_n\}_{n\geq 0}$ as in Proposition \ref{minborne}, we have

\begin{equation}\label{e6}
0<\lambda\int_{\Omega}|u_n|^q~dx-
\int_{\Omega}|\nabla u_n|^q~dx=\int_{\Omega}|\nabla u_n|^p~dx\rightarrow 0, \text{as $n\rightarrow\infty$}.
\end{equation}
By Proposition \ref{minborne}, we deduce that $\{u_n\}_{n\geq 0}$ is bounded in $W^{1,q}_0(\Omega).$ Therefore there exists $u_0\in W^{1,q}_0(\Omega)$ such that $u_n \rightharpoonup u_0$ in $W^{1,q}_0(\Omega)\subset W^{1,p}_0(\Omega)$ and $u_n\rightarrow u_0$ in $L^q(\Omega).$

Thus $\displaystyle{\int_{\Omega}}|\nabla u_0|^p\leq \lim_{n\rightarrow\infty}\inf\displaystyle
{\int_{\Omega}}|\nabla u_n|^p~dx=0.$ Consequently  $u_0=0$, $u_n \rightharpoonup 0$ in $W^{1,q}_0(\Omega)\subset W^{1,p}_0(\Omega)$ and $u_n\rightarrow 0$ in $L^q(\Omega).$ Writing again $w_n=\frac{
u_n}{\|u_n\|_q}$ we have
 $$0<\frac{\lambda\displaystyle{\int_{\Omega}}|u_n|^q~dx-\int_{\Omega}|\nabla u_n|^q~dx}{\|u_n\|^q_{q}}=\|u_n\|^{p-q}_{q} \displaystyle{\int_{\Omega}}|\nabla w_n|^p~dx,
$$ 
and
\begin{eqnarray*}
\int_{\Omega}|\nabla w_n|^p~dx
=\|u_n\|^{q-p}_{q}\Big(\lambda-
\displaystyle{\int_{\Omega}}|\nabla w_n|^2~dx\Big)\rightarrow 0~ \text{as $n\rightarrow \infty$},
\end{eqnarray*}
since $\|u_n\|_{q}\rightarrow 0$ and $p<q$, $\{w_n\}_{n\geq 0}$ is bounded in $W_0^{1,q}(\Omega).$ Next since $w_n \rightharpoonup w_0$, we deduce that $\displaystyle{\int_{\Omega}}|\nabla w_0|^p~dx\leq \lim_{n\rightarrow\infty}\inf\displaystyle
{\int_{\Omega}}|\nabla w_n|^p~dx=0$ and we have $w_0=0.$
This is a contradiction since $\|w_n\|_{q}=1$ for each $n$. Thus $m>0$.
\end{proof}

\begin{proposition}\label{minatteind}
 There exists $u\in \mathcal{N}_{\lambda}$ such that
$E_{\lambda}(u)=m.$
\end{proposition}
\begin{proof}
Let $\{u_n\}_{n\geq 0}\subset\mathcal{N}_{\lambda}$ be a minimizing sequence, i.e., $E_{\lambda}(u_n)\rightarrow m$ as $n\rightarrow\infty.$ Thanks to Proposition \ref{minborne}, we have that $\{u_n\}$ is bounded in $W_0^{1,q}(\Omega).$ It follows that there exists $u_0\in W_0^{1,q}(\Omega)$ such that $u_n\rightharpoonup u_0$ in $W_0^{1,q}(\Omega)\subset W_0^{1,p}(\Omega)$ and strongly in $L^q(\Omega).$ The results in the two propositions above guarantee that $E_{\lambda}(u_0)\leq \displaystyle{\lim_{n\rightarrow\infty}}\inf E_{\lambda}(u_n)=m.$ Since for each $n$ we have $u_n\in\mathcal{N}_{\lambda}$, then
\begin{equation}\label{att1}
\int_{\Omega}|\nabla u_n|^q~dx+\int_{\Omega}|\nabla u_n|^p~dx=\lambda \int_{\Omega}|u_n|^q~dx~~~\text{for all $n$.}
\end{equation}
Assuming $u_0\equiv 0$ on $\Omega$ implies that $ \displaystyle{\int_{\Omega}}|u_n|^q~dx\rightarrow 0$ as $n\rightarrow \infty$, and by relation $(\ref{att1})$ we obtain that $\displaystyle{\int_{\Omega}}|\nabla u_n|^q~dx\rightarrow 0$ as $n\rightarrow \infty.$ Combining this with the fact that $u_n$ converges weakly to $0$ in $W_0^{1,q}(\Omega)$, we deduce that $u_n$ converges strongly to $0$ in $W_0^{1,q}(\Omega)$ and consequently in $W_0^{1,p}(\Omega)$.
Hence we infer that \begin{eqnarray*}
\lambda\int_{\Omega}|u_n|^q~dx-
\int_{\Omega}|\nabla u_n|^q~dx=\int_{\Omega}|\nabla u_n|^p~dx\rightarrow 0, \text{as $n\rightarrow\infty$}.
\end{eqnarray*}
Next, using similar argument as the one used in the proof of Proposition \ref{minpositive}, we will reach to a contradiction, which shows that $u_0\not\equiv 0.$ Letting $n\rightarrow \infty$ in relation (\ref{att1}), we deduce that 
\begin{eqnarray*}
\int_{\Omega}|\nabla u_0|^q~dx+\int_{\Omega}|\nabla u_0|^p~dx\leq\lambda \int_{\Omega}|u_0|^q~dx.
\end{eqnarray*}
If there is equality in the above relation then $u_0\in\mathcal{N}_{\lambda}$ and $m\leq E_{\lambda}(u_0)$. Assume by contradiction that
\begin{equation}\label{att2}
\int_{\Omega}|\nabla u|^q~dx+\int_{\Omega}|\nabla u|^p~dx<\lambda \int_{\Omega}|u|^q~dx.
\end{equation} 
Let $t>0$ be such that $tu_0\in\mathcal{N}_{\lambda},$ i.e.,$$t=\Bigg(\frac{\lambda\displaystyle{\int_{\Omega}}|u_0|^q~dx-
\displaystyle{\int_{\Omega}}|\nabla u_0|^q~dx}{\displaystyle{\int_{\Omega}}|\nabla u_0|^p~dx}\ \Bigg)^{\frac{1}{p-q}}
.$$
We note that $t\in(0,1)$ since $1<t^{p-q}$ (using (\ref{att2})). Finally, since $tu_0\in \mathcal{N}_{\lambda}$ with $t\in(0,1)$ we have 
\begin{eqnarray*}
0<m\leq E_{\lambda}(tu_0)&=&\left(\frac{1}{p}-\frac{1}{q}\right)\int_{\Omega}|\nabla(tu_0)|^p~dx=t^p\left(\frac{1}{p}-\frac{1}{q}\right)\int_{\Omega}|\nabla u_0|^p~dx\\
&=&t^p E_{\lambda}(u_0)\\
&\leq & t^p\lim_{k\rightarrow\infty}\inf E_{\lambda}(u_k)=t^p m<m~\text{for $t\in(0,1)$,}
\end{eqnarray*}
and this is a contradiction which assures that relation (\ref{att2}) cannot hold and consequently we have $u_0\in \mathcal{N}_{\lambda}$. Hence $m\leq E_{\lambda}(u_0)$ and $ m= E_{\lambda}(u_0)$.
\end{proof}
\begin{theorem}
Let $p<q.$ Then every $\lambda\in (\lambda_1^D(q),\infty)$ is a first eigenvalue of problem (\ref{e1}).
\end{theorem}
\begin{proof}
Let $u\in \mathcal{N}_{\lambda}$ be such that $E_{\lambda}(u)=m$ (thanks to Proposition \ref{minatteind}). We show that $\langle E'_{\lambda}(u), v\rangle=0$ for all $v\in W^{1,q}_0(\Omega).$ We recall that for $u\in \mathcal{N}_{\lambda}$, we have \begin{eqnarray*}
\int_{\Omega}|\nabla u|^q~dx+\int_{\Omega}|\nabla u|^p~dx=\lambda \int_{\Omega}|u|^q~dx.
\end{eqnarray*} 
Let $v\in W^{1,q}_0(\Omega).$ For every $\delta$ in some small interval $(-\varepsilon,\varepsilon)$ certainly the function $u+\delta v$ does not vanish identically. Let $t(\delta)>0$ be a function such that $t(\delta)(u+\delta v)\in  \mathcal{N}_{\lambda},$ namely 
$$t(\delta)=\Bigg(\frac{\lambda\displaystyle{\int_{\Omega}}|u+\delta v|^q~dx-
\int_{\Omega}|\nabla(u+\delta v)|^q~dx}{\displaystyle{\int_{\Omega}}|\nabla(u+\delta v)|^p~dx}\ \Bigg)^{\frac{1}{p-q}}
.$$ The function $t(\delta)$ is a composition of differentiable functions, so it is differentiable. The precise expression of $t'$ does not matter here. Observe that $t(0)=1.$ The map $\delta\mapsto t(\delta)(u+\delta v) $
defines a curve on $\mathcal{N}_{\lambda}$ along which we evaluate $E_{\lambda}.$ Hence we define $\gamma: (-\varepsilon,\varepsilon)\rightarrow \mathbb{R}$ as $\gamma(\delta)=E_{\lambda}(t(\delta)(u+\delta v)).$ By construction, $\delta=0$ is a minimum point for $\gamma.$ Consequently $$0=\gamma'(0)=\langle E'_{\lambda}(t(0)u), t'(0)u+t(0)v\rangle=t'(0)\langle E'_{\lambda}(u),u\rangle+\langle E'_{\lambda}(u),v\rangle=\langle E'_{\lambda}(u),v\rangle$$ using the fact that $\langle E'_{\lambda}(u), u\rangle=0$ because $u\in \mathcal{N}_{\lambda}.$ We then obtained that $\langle E'_{\lambda}(u),v\rangle=0$ for all $v\in W^{1,q}_0(\Omega).$
\end{proof}

\par \bigskip
\section{Bifurcation}\label{S3}
In this section we discuss bifurcation phenomena for problem \eqref{e1}. We begin with the following.
\begin{definition}\label{bifurcation-theorem}
A real number $\mu $ is called a bifurcation point of (\ref{e1}) if and only if there is a sequence $(u_n,\mu_n)$ of solutions of (\ref{e1}) such that $u_n\not\equiv 0$ and 
$$\mu_n\rightarrow\mu, \ \|u_n\|_{1,s}\rightarrow 0,~~~\text{as}~~n\rightarrow\infty,~~s=p\ (\text{if }\ p > q),~~\text{or}\ \ s = q\ (\text{if }\ p < q).\
$$
\end{definition}
\par \medskip \noindent
{\bf Observations:}
Define $F: W^{1,p}_0(\Omega)\cap  W^{1,q}_0(\Omega)\rightarrow\mathbb{R}$ by $$F(u)=\frac{\frac{1}{p}\int_{\Omega}|\nabla u|^pdx+\frac{1}{q}\int_{\Omega}|\nabla u|^qdx}{\frac{1}{q}\int_{\Omega}|u|^qdx},~~\text{for all}~~u\in  W^{1,p}_0(\Omega)\cap  W^{1,q}_0(\Omega).$$
By setting $u=re_1$, where $e_1$ stands as the normalized eigenfunction associated to the eigenvalue $\lambda^D_1(q)$ of the $q$-homogenous equation \eqref{dirichletlap}, we then have
$$F(re_1)=\frac{\frac{r^{p-q}}{p}\int_{\Omega}|\nabla e_1|^pdx+\frac{1}{q}\int_{\Omega}|\nabla e_1|^qdx}{\frac{1}{q}\int_{\Omega}|e_1|^qdx}.$$
We distinguish two cases:
\begin{enumerate}
    \item Assume that $p>q.$ Thus we find that $F(re_1)\rightarrow\lambda^D_1(q)$ as $r\rightarrow 0,$ which indicates bifurcation in $0$ from $\lambda^D_1(q).$
    \item Assume that $p<q.$  We find that $F(re_1)\rightarrow\infty$ as $r\rightarrow 0,$ which indicates there is no bifurcation in $0$ from $\lambda^D_1(q).$ One is lead to look for bifurcation at infinity.
\end{enumerate}
\par \bigskip \noindent
Our aim is to show that the variational $q$-homogenous eigenvalues $\lambda^D_k(q)$ of equation \eqref{dirichletlap} are bifurcation points for the nonlinear eigenvalues $\lambda^D_k(p,q;\rho)$ of equation \eqref{e1}. More precisely, we will show that
$$ \lambda^D_k(p,q;\rho) \to \lambda^D_k(q) \ \ \text{as}\ \rho \to 0
.$$
\par \medskip
As in section \ref{S1}, let $D_{\rho}(p,q)=\{u\in W^{1,p}_0(\Omega)\setminus\{0\}\subset W^{1,q}_0(\Omega)\setminus\{0\}: \int_{\Omega}|u|^qdx=\rho\}$ and
$$\Gamma_{k,\rho}=\{A\subset D_{\rho}(p,q): A \ \text{symmetric}, A\ \text{compact}, \gamma(A)\geq k\}.
$$ 
By the definition of $\lambda_k^D(q)$ we know that for $\eps>0$ small there is $A_{\eps}\in\Gamma_{k,1}$ such that 
$$
\sup_{\{u\in A_{\eps},~\int_{\Omega}|u|^qdx=1\}}\int_{\Omega}|\n u|^qdx\leq \lambda^D_k(q) +\eps\ .
$$
We want to approximate $A_{\eps}$ by a finite-dimensional set. Since $A_{\eps}$ is compact, for every $\delta>0$ there exist a finite number of points $x_1,\dots,x_{n(\delta)}$ such that
\begin{equation}\label{Clarkn1}
    A_{\eps}\subset\bigcup_{i=1}^{n(\delta)}B_{\delta}(x_i).
\end{equation}
Let $E_n=\text{span}\{x_1,\dots,x_{n(\delta)}\},$ and set 
\begin{equation}\label{choose1}
    P_nA_{\eps}:=\{P_nx,~~x\in A_{\eps}\},
\end{equation}
where $P_nx\in E_n$ is such that $$\|x-P_nx\|_{1,q}=\inf\{\|x-z\|_{1,q},~~z\in E_n\}.$$ We claim that $\gamma(P_n A_{\eps})\geq k.$ Clearly, $P_n A_{\eps}$ is symmetric and compact. Furthermore, $0\not\in P_n A_{\eps}$; indeed since $A_{\eps}$ is compact, and $0\not\in A_{\eps}$, there is small ball $B_{\tau}(0)$ such that $A_{\eps}\cap B_{\tau}(0)=\emptyset.$ Now, choose $\delta>0$ in (\ref{Clarkn1}) such that $\delta<\tau/2.$ Then, for $x\in A_{\eps}$ there is $x_i\in E_n$, for some $i \in \{1,\dots,n(\delta)\}$, such that $\|x-x_i\|_{1,q}<\delta$, and hence 
$$\|x-P_n x\|_{1,q}=\inf\{\|x-z\|_{1,q},~~z\in E_n\}\leq \|x-x_i\|_{1,q}<\tau/2$$ and thus $P_nA_{\eps}\cap B_{\tau/2}(0)=\emptyset.$\\
Finally, we have to show that $\gamma(P_nA_{\eps})\geq k.$ This is again by approximation: since $\gamma(A_{\eps})\geq k,$ there exist a continuous and odd map $g: A_{\eps}\rightarrow \mathbb{R}^k\setminus\{0\}.$ Then by Tietze extension theorem there exist a continuous and odd map $\tilde{g}: W^{1,q}_0(\Omega)\rightarrow \mathbb{R}$ such that $\tilde{g}_{|A_{\eps}}=g.$ By continuity and compactness of $A_{\eps}$ we can conclude that\\ $\tilde{g}_{|P_nA_{\eps}}: W^{1,q}_0(\Omega)\rightarrow \mathbb{R}^k\setminus\{0\}.$ Now, again by approximation, we conclude that there is a $n = n(\eps)$ such that 
$$
\sup_{\{u\in P_nA_{\eps}\}}\int_{\Omega}|\n u|^qdx\leq \lambda^D_k(q) + 2\eps\ 
.$$
Finally, note that by homogeneity
$$
\inf_{A\in\Gamma_{k,\rho}}\sup_{u\in A}\int_{\Omega}|\nabla u|^q~dx=\lambda^D_k(q)\,\rho
$$
and hence also
\begin{equation}\label{PnA}
\sup_{\{u\in \rho\,P_nA_{\eps}\}}\int_{\Omega}|\n u|^qdx\leq \big(\lambda^D_k(q) + 2\eps \big)\,\rho.
\end{equation} 
%We also set 
%$$\lambda^D_k(p,q)=\inf_{A\in\Gamma_{k,1}}\sup_{u\in A}\{\frac{q}{p}\int_{\Omega}|\nabla u|^p~dx+\int_{\Omega}|\nabla u|^q~dx\}.$$
\par \bigskip
Recall that by \eqref{ck} we have, for each integer $k > 0$,
$$
c_k(p,q;\rho)=\inf_{A\in\Gamma_{k,\rho}}\sup_{u\in A}\Big\{\frac{1}{p}\int_{\Omega}|\nabla u|^p~dx+\frac 1q\int_{\Omega}|\nabla u|^q~dx\Big\}
$$
We first prove the following lemma which is useful for the bifurcation result from zero.
\begin{lemma}\label{avanbifurc}
Let $p>q$. For any integer $k>0$ and  $\rho>0,$ $\eps>0,$ there exists a positive constant $C(\eps)$ such that the following estimate holds: 
$$
|c_k(p,q;\rho)- \frac 1q \, \lambda^D_k(q)\, \rho|\leq C(\eps)\rho^{p/q} + 2\eps \, \rho
.$$
\end{lemma}
\begin{proof}
For any $k>0,$ we clearly have $c_k(p,q;\rho)\geq \frac 1q\, \lambda^D_k(q)\, \rho.$ 
\par \smallskip \noindent
By \eqref{PnA} we can estimate 
\begin{align*}
    c_k(p,q,\rho)&=\inf_{A\in\Gamma_{k,\rho}}\sup_{u\in A}\Big\{\frac{1}{p}\int_{\Omega}|\nabla u|^p~dx+\frac 1q\int_{\Omega}|\nabla u|^q~dx\Big\}\\
    &\leq  \sup_{u\in \rho\, P_nA_{\eps}}\Big\{\frac{1}{p}\int_{\Omega}|\nabla u|^p~dx+\frac 1q\int_{\Omega}|\nabla u|^q~dx\Big\}\\
    &\leq \sup_{u\in \rho\,P_nA_{\eps}}\frac{1}{p}\int_{\Omega}|\nabla u|^p~dx+ \sup_{u\in \rho\, P_nA_{\eps}}\frac 1q\int_{\Omega}|\nabla u|^q~dx\\
    &\leq \frac{1}{p}\int_{\Omega}|\nabla v|^p~dx+\frac 1q(\lambda^D_k(q)+2\eps)\rho
\end{align*}
for some $v\in \rho\,P_nA_{\eps}$ with $\int_{\Omega}|v|^qdx=\rho.$ Since $ P_nA_{\eps}$ is finite-dimensional, there exists a positive constant $ C(\eps)$ such that
$$
\big(\int_{\Omega}|\n v|^p dx\big)^{1/p} \leq C(\eps) \big(\int_{\Omega}|v|^qdx\big)^{1/q} 
$$ 
and hence
$$
\int_{\Omega}|\n v|^p dx \leq C(\eps) \big(\int_{\Omega}|v|^qdx\big)^{p/q} = C(\eps)\, \rho^{p/q}
.$$
Finally, we get
$$
0\leq c_k(p,q;\rho)- \frac 1q\,\lambda^D_k(q)\, \rho\leq C(\eps)\rho^{p/q}+2\eps\rho
.$$ 
%that is $$0\leq \lambda^D_k(p,q,\theta)\theta^{-q}-\lambda^D_k(q)\leq C_n(\eps)\theta^{p-q}+\eps. $$
\end{proof}
\par \bigskip
\subsection{Bifurcation from zero}$ $
\par \smallskip \noindent
Here, we show that for equation (\ref{e1}), for $p > q$, there is a branch of first eigenvalues bifurcating from $(\lambda_k^D(q), 0)\in \mathbb{R}^+\times W^{1,p}_0(\Omega)$.
\begin{theorem}\label{bifurcation-from-zero}
Let $1<q<p<\infty.$ Then for each integer $k>0$ the pair $(\lambda^D_k(q), 0)$ is a bifurcation point of problem (\ref{e1}).
\end{theorem}
\noindent
An illustration of the bifurcation results obtained in Theorem \ref{bifurcation-from-zero} is given by the figure \ref{fig1} below.
\begin{proof}
We aim to show that $ \lambda^D_k(p,q;\rho)\rightarrow\lambda^D_k(q)$ and  $\|u_k\|_{1,p}\rightarrow 0$, as $\rho\rightarrow 0^+.$  Thanks to Lemma \ref{avanbifurc} we have  
$$
\frac 1p \int_\Omega|\nabla u_k|^pdx   
\le C_n(\eps)\rho^{p/q}+2\eps\, \rho 
$$
Furthermore
$$
\begin{array}{ll}
0 &\le   \lambda^D_k(p,q;\rho)\, \rho - \lambda^D_k(q)\rho
\vspace{0.2cm}\\
&=\dis \int_\Omega |\nabla u_k|^pdx + \int_\Omega |\nabla u_k|^qdx
- \lambda_k^D(q) \rho
\vspace{0.2cm}\\
&
=\dis \frac qp \int_\Omega |\nabla u_k|^pdx + \int_\Omega |\nabla u_k|^qdx  -    \lambda_k^D(q) \rho + (1-\frac qp)\int_\Omega |\nabla u_k|^pdx
\vspace{0.2cm}\\
&
= \dis q \, c_k(p,q;\rho) - \lambda_k^D(q)\rho + (1-\frac qp)\int_\Omega |\nabla u_k|^pdx
\vspace{0.2cm}\\
&
\le  C\, \big(C_n(\eps)\rho^{p/q}+2\eps\, \rho \big) 
\end{array}
$$  
Since $\eps > 0$ is arbitrary we get the first claim.
\par \smallskip
Let us prove that  $\|u_k\|_{1,p}\rightarrow 0$ as $\rho\rightarrow 0^+.$ Letting $v=u_k$ in relation (\ref{e2}), we have 
$$
    \int_{\Omega}|\nabla u_k|^{p}~dx+\int_{\Omega}|\nabla u_k|^{q}~dx=\lambda^D_k(p,q;\rho)\int_{\Omega}|u_k|^{q}~dx.
$$ 
Therefore 
$$        \int_{\Omega}|\nabla u_k|^{p}~dx \le \lambda^D_k(p,q;\rho)\int_{\Omega}|u_k|^{q}~dx\\
        \leq C_k\ \rho
$$
Hence $\displaystyle{\int}_{\Omega}|\nabla u_k|^{p}~dx\rightarrow 0$ as $\rho\rightarrow 0.$ This complete the proof.
%We have 
%$$\int_{\Omega}|\nabla u_k|^pdx=(\lambda^D_k(p,q)-\lambda^D_k(q))\int_{\Omega}|u_k|^qdx\leq C \rho^{p-q}\int_{\Omega}|u_k|^qdx$$
%Since $p>q,$ there exists $M>0$ such that $\int_{\Omaga}|\nabla u_k|^pdx\leq M\|u_k\|_{1,p}$
%\todo{To be continued}
\end{proof}
\par \medskip
\subsection{Bifurcation from infinity}$ $
\par \smallskip \noindent
The goal is to prove that if $p<q,$ there is a branch of first eigenvalues bifurcating from $(\lambda_k^D(q), \infty).$\\
For $u\in W^{1,q}_0(\Omega),~u\neq 0,$ we set $w=u/\|u\|_{1,q}^2$.  We have $\|w\|_{1,q}=\frac{1}{\|u\|_{1,q}}$ and 
$$\begin{array}{ll}
|\nabla w|^{p-2}\nabla w&=\frac{1}{\|u\|_{1,q}^{2(p-1)}} |\nabla u|^{p-2}\nabla u,~|\nabla w|^{q-2}\nabla w
\vspace{0.2cm}\\
&=\frac{1}{\|u\|_{1,q}^{2(q-1)}} |\nabla u|^{q-2}\nabla u,~|w|^{q-2} w=\frac{1}{\|u\|_{1,q}^{2(p-1)}} |u|^{q-2} u
\end{array}
$$
Introducing this change of variable in (\ref{e2}), we find that, 
\begin{equation*}\label{em3}
\|u\|_{1,q}^{2(p-q)}\int_{\Omega}|\nabla w|^{p-2}\nabla w\cdot\nabla v~dx+\int_{\Omega}|\nabla w|^{q-2}\nabla w\cdot\nabla v~dx=\lambda\int_{\Omega}|w|^{q-2}w~v~dx
\end{equation*}
for every $v\in W^{1,q}_0(\Omega)$.
This leads to the following nonlinear eigenvalue problem (for $1<p<q<\infty$)
\begin{equation}\label{em5}
\left\{
\begin{array}{rll}
-\|w\|_{1,q}^{2(q-p)}\Delta_p w-\Delta_q w &=\lambda |w|^{q-2}w~~~~&\text{in $\Omega$} \vspace{0.2cm}\\
w &= \displaystyle 0~~~~~~~~~~~~~~~~&\text{on $\partial\Omega$}.
\end{array}
\right.
\end{equation}
\begin{proposition}
Assume that $p<q.$ If $(\lambda,0)$ is a bifurcation point of solutions of problem (\ref{em5}) then $\lambda$ is an eigenvalue of problem (\ref{dirichletlap}).
\end{proposition}
\begin{proof}
Since $(\lambda,0)$ is a bifurcation point from zero of solutions of problem (\ref{em5}), there is a sequence of nontrivial solutions of problem (\ref{em5}) such that $\lambda_n\rightarrow\lambda$ and $\|w_n\|_{1,q}\rightarrow 0$ in $W^{1,q}_0(\Omega).$ We then have 
\begin{equation}\label{pp1}
   \|w_n\|_{1,q}^{2(q-p)}\int_{\Omega}|\nabla w_n|^{p-2}\nabla w_n\cdot\nabla v~dx+\int_{\Omega}|\nabla w_n|^{q-2}\nabla w_n\cdot\nabla v~dx=\lambda_n\int_{\Omega}|w_n|^{q-2}w_nv~dx.
\end{equation}
By using the argument in Remark \ref{r1} and then passing to limit, we complete the proof.
\end{proof}
Let us consider a small ball $B_r(0) :=\{~w~\in W^{1,q}_0(\Omega)\setminus\{0\}/~~~\|w\|_{1,q}< r~\},$ and
 the operator $$T :=-\|\cdot\|_{1,q}^{2(q-p)}\Delta_p-\Delta_q : W^{1,q}_0(\Omega)\subset W^{1,p}_0(\Omega)\rightarrow W^{-1,p'}(\Omega)\subset W^{-1,q'}(\Omega).$$
\begin{proposition}\label{invert}
Let $1<p<q.$ There exists $r>0$ such that the mapping \\$T : B_r(0)\subset W^{1,q}_0(\Omega)\rightarrow W^{-1,q'}(\Omega)$ is strongly monotone, i.e., there exists $C>0$ such that $$\langle T(u)-T(v), u-v\rangle\geq C\|u-v\|^q_{1,q}, ~~\text{for}~~u,v\in B_r(0)\subset W^{1,q}_0(\Omega)$$ with $r>0$ sufficiently small.
\end{proposition}
\begin{proof}
Using that $-\Delta_p$ is strongly monotone on $W^{1,p}_0(\Omega)$ on the one hand and the H\"older inequality on the other hand, we have
\begin{eqnarray}\label{bf1}\notag
\langle T(u)-T(v), u-v\rangle&=&\|\nabla u-\nabla v\|_{q}+\left(\|u\|_{1,q}^{2(q-p)}(-\Delta_pu)-\|v\|_{1,q}^{2(q-p)}(-\Delta_pv), u-v \right)\\ \notag
&=&\|u\|_{1,q}^{2(q-p)}\left((-\Delta_pu)-(-\Delta_p v), u-v\right)\\ \notag
&+& \left(\|u\|_{1,q}^{2(q-p)}-\|v\|_{1,q}^{2(q-p)} \right)\left(-\Delta_pv, u-v\right)\\ \notag
&\geq &\|\nabla u-\nabla v\|_{q} -\left|\|u\|_{1,q}^{2(q-p)}-\|v\|_{1,q}^{2(q-p)}\right|\|\nabla v\|_{p}^{p-1}\|\nabla (u-v)\|_{p}\\
&\geq&\|\nabla u-\nabla v\|_{q} -\left|\|u\|_{1,q}^{2(q-p)}-\|v\|_{1,q}^{2(q-p)}\right|C\| v\|_{1,q}^{p-1}\| u-v\|_{1,q}.
\end{eqnarray}
By the Mean Value Theorem, we obtain that there exists $\theta\in [0,1]$ such that
\begin{eqnarray*}
\left|\|u\|_{1,q}^{2(q-p)}-\|v\|_{1,q}^{2(q-p)}\right|&=&\left|\frac{d}{dt}\left(\|u+t(v-u)\|^2_{1,q}\right)^{q-p}|_{t=\theta} (v-u)\right|\\
&=&\left|(q-p)\left(\|u+\theta(v-u)\|^2_{1,q}\right)^{q-p}2\left(u+\theta(v-u),v-u\right)_{1,q}\right|\\
&\leq& 2(q-p)\|u+\theta(v-u)\|^{q-p}_{1,q}\|u+\theta(v-u)\|_{1,q}\|u-v\|_{1,q}\\
&=&2(q-p)\|u+\theta(v-u)\|_{1,q}^{2q-p}\|u-v\|_{1,q}\\
&\leq & 2(q-p)\left((1-\theta)\|u\|_{1,q}+\theta\|v\|_{1,q}\right)^{2q-p}\|u-v\|_{1,q}\\
&\leq & 2(q-p) r^{2q-p}\|u-v\|_{1,q}.
\end{eqnarray*}
Thus, continuing with the estimate of equation (\ref{bf1}), we get
$$\langle T(u)-T(v), u-v\rangle\geq \|u-v\|^q_{1,q}-2(q-p) r^{2q-1}C\|u-v\|^2_{1,q},$$ 
and hence, for $r\to 0$ we end the proof.
\end{proof}
We first show the existence of variational eigenvalues of the nonlinear equation (\ref{em5}). 
\begin{theorem}\label{manysolutions}
Let $1<p<q$ be given. Then, for a fixed $\rho>0,$ there exists a non-decreasing sequence of eigenvalues $\tilde{\lambda}^D_k(p,q;\rho),$ with corresponding eigenfunctions $w_k(p,q;\rho)\in W^{1,q}_0(\Omega) $ for the nonlinear eigenvalue problem (\ref{em5}).
\end{theorem}
We again rely on \cite[Proposition 10.8]{AM} for the proof of Theorem \ref{manysolutions}.
\begin{proof}
 Let $O_{\rho}(p,q)=\{w\in W^{1,q}_0(\Omega)~:~\int_{\Omega}|w|^qdx=\rho\}$, and $\Sigma_{k,\rho}(p,q)=\{A\subset\Sigma~:~\gamma(A\cap O_{\rho}(p,q))\geq k\},$ where $\Sigma=\{A\subset W^{1,q}_0(\Omega):~~A~~\text{closed},~~ A=-A\}.$ 
 Set 
 \begin{equation}\label{dk}
 d_k(p,q;\rho)=\inf_{A\in\Sigma_{k,\rho}(p,q)}\sup_{u\in A}\left(\frac{q}{p}\|w\|_{1,q}^{2(q-p)}\int_{\Omega}|\nabla w|^p~dx+\int_{\Omega}|\nabla w|^q~dx\right)>0.
 \end{equation} 
We show that:
\begin{enumerate}
    \item the functional $F(w)=\displaystyle{\frac{q}{p}\|w\|_{1,q}^{2(q-p)}\int_{\Omega}|\nabla w|^p~dx+\int_{\Omega}|\nabla w|^q~dx}$ satisfies the (PS) condition on $O_{\rho}(p,q)$, and \vspace{0.2cm}
    \item if $d=d_k(p,q)=\dots=d_{k+m-1}(p,q),$ then the set $K_d$ of critical points of $I$ at the critical level $d$ has a genus $\gamma(K_d)\geq m.$
\end{enumerate}
We prove (i). Let $\{w_j\}\subset O_{\rho}(p,q)$ a (PS) sequence, i.e, for all $j,$ $M>0$ $|F(w_j)|\leq M$ and $F'(w_j)\rightarrow 0$ in $W^{-1,q'}(\Omega)$ as $j \rightarrow \infty.$ We first show that $\{w_j\}$ is bounded in $O_{\rho}(p,q)\subset W^{1,p}_0(\Omega).$ Since $w_j\in W^{1,q}_0(\Omega),$ with the Poincar\'e inequality, we have $\int_{\Omega}|w_j|^qdx\leq C\int_{\Omega}|\nabla w_j|^qdx $ and it follows that 
\begin{align*}
    M\geq |F(w_j)|&\geq \frac{q}{p}\|w_j\|_{1,q}^{2(q-p)}\int_{\Omega}|\nabla w_j|^pdx+\frac{1}{C}\int_{\Omega}|w_j|^qdx\\
    &\geq \|w_j\|_{1,p}^{2q-p}+\frac{\rho}{C},~~\text{since}~~W^{1,q}_0(\Omega)\subset W^{1,p}_0(\Omega).
\end{align*}
Then $\{w_j\}$ is bounded in $O_{\rho}(p,q)\subset W^{1,q}_0(\Omega).$ We can assume that up to a subsequence still denoted $\{w_j\}$, there exists $w\in O_{\rho}(p,q)\subset W^{1,q}_0(\Omega)$ such that $w_j \rightharpoonup w$ in $O_{\rho}(p,q)\subset W^{1,q}_0(\Omega).$ Now, we show that $w_j$ converges strongly to $w$ in $O_{\rho}(p,q)\subset W^{1,q}_0(\Omega).$ Since $F'(w_j)\rightarrow 0$ in $W^{-1,q'}(\Omega)$ as $j\rightarrow \infty,$ we have $F'(w_j)(w_j-w)\rightarrow 0$ and $F'(w)(w_j-w)\rightarrow 0$ as $j\rightarrow\infty.$
 %Consider the $C^1$ energy functional associated to equation (\ref{em5}),\\
%$F_{\lambda}: B_r(0)\subset W^{1,q}_0(\Omega)\rightarrow\mathbb{R}$ $$F_{\lambda}(w)=\frac{q}{p}\|w\|_{1,q}^{2(q-p)}\int_{\Omega}|\nabla w|^p~dx+\int_{\Omega}|\nabla w|^q~dx-\lambda\int_{\Omega}|w|^q~dx.$$
 We have 
 \begin{equation*}
     \begin{split}
         \langle F'(w_j)- &F'(w), w_j-w\rangle
         \\&= q\int_{\Omega}\left(\|w_j\|_{1,q}^{2(q-p)}|\nabla w_j|^{p-2}\nabla w_j-\|w\|_{1,q}^{2(q-p)}|\nabla w|^{p-2}\nabla w\right)\cdot \nabla(w_j-w)dx\\
         &+q\int_{\Omega}\left(|\nabla w_j|^{q-2}\nabla w_j-|\nabla w|^{q-2}\nabla w\right)\cdot \nabla(w_j-w)~dx.
     \end{split}
 \end{equation*}
Thanks to Proposition \ref{invert}, it follows that
 \begin{equation*}
     \begin{split}
         \langle F'(w_j)- F'(w), w_j-w\rangle &\geq C \|w_j-w\|^q_{1,q}.
     \end{split}
 \end{equation*}
 Therefore $\|w_j-w\|_{1,q}\rightarrow 0$ as $j\rightarrow +\infty$ and $w_j$ converges strongly to $w$ in $W^{1,q}_0(\Omega).$
 \par \medskip
 
  The proof of (ii) is similar to the last part of the proof of Theorem \ref{sequence1}.
 %\todo{To check carefully till the end}
\end{proof}
\par \medskip
\begin{theorem}\label{bifurcation-from-infinity}
Let $p<q.$ Then for each integer $k>0$ the pair $(\lambda^D_k(q,\rho), \infty)$ is a bifurcation point of problem (\ref{e1}). 
\end{theorem}
The proof of Theorem \ref{bifurcation-from-infinity} will follow immediately from the following remark, and the proof that $(\lambda_k^D(q,\rho),0)$ is a bifurcation point of (\ref{em5}), which will be shown in Theorem \ref{binfnon} below.

\begin{remark}\label{RMK}
With the change of variable, we have that the pair $(\lambda^D_k(q,\rho),\infty)$ is a bifurcation point for the problem (\ref{e1}) if and only if the pair $(\lambda^D_k(q,\rho),0)$ is a bifurcation point for the problem (\ref{em5}).
\end{remark}
\par \medskip
Before we proceed to the proof of Theorem \ref{binfnon} below, we show the following lemma.
\begin{lemma}\label{bifurcainfinity}
Let $1<p<q<\infty$. For any integer $k>0$ and  $\rho>0,$ $\eps>0,$ there exists a positive constant $D(\eps)$ such that the following estimate holds: $$|d_k(p,q;\rho)-\lambda^D_k(q,\rho)|\leq (D(\eps)+\eps)\rho^{\frac{2q-p}{p}}
$$ 
where $d_k(p,q;\rho)$ is given by \eqref{dk}, and 
$\dis \lambda^D_k(q,\rho)=\inf_{A\in\Gamma_{k,\rho}}\sup_{u\in A}\int_{\Omega}|\nabla u|^q~dx=\lambda^D_k(q)\rho\, .
$
\end{lemma}

\begin{proof}
For any $k>0,$ we clearly have $d_k(p,q;\rho)\geq \lambda^D_k(p,\rho).$ As in (\ref{choose1}), we choose $P_nA_{\eps}$ such that 
$$\sup_{\{w\in P_nA_{\eps},~\int_{\Omega}|w|^qdx=1\}}\int_{\Omega}|\n w|^qdx\leq \lambda^D_k(q,\rho)+\eps
$$ and so 
$$
\sup_{\{w\in P_nA_{\eps,\rho},~\int_{\Omega}|w|^qdx=\rho\}}\int_{\Omega}|\n w|^qdx\leq (\lambda^D_k(q,\rho)+\eps)\rho,
$$ 
where $P_nA_{\eps,\rho}=\{w\in P_nA_{\eps}:~~\int_{\Omega}|w|^qdx=\rho\}.$
Then 
\begin{align*}
    d_k(p,q;\rho)&=\inf_{A\in\Gamma_{k,\rho}}\sup_{u\in A}\Big\{\frac{q}{p}\|w\|_{1,q}^{2(q-p)}\int_{\Omega}|\nabla w|^p~dx+\int_{\Omega}|\nabla w|^q~dx\Big\}\\
    &\leq  \sup_{u\in P_nA_{\eps,\rho}}\Big\{\frac{q}{p}\|w\|_{1,q}^{2(q-p)}\int_{\Omega}|\nabla w|^p~dx+\int_{\Omega}|\nabla w|^q~dx\Big\}\\
    &\leq \sup_{u\in P_nA_{\eps,\rho}}\frac{q}{p}\|w\|_{1,q}^{2(q-p)}\int_{\Omega}|\nabla w|^p~dx+ \sup_{w\in P_nA_{\eps,\rho}}\int_{\Omega}|\nabla w|^q~dx\\
    &\leq\frac{q}{p}\|v\|_{1,q}^{2(q-p)}\int_{\Omega}|\nabla v|^q~dx+(\lambda^D_k(q)+\eps)\rho~~\text{since}~~p<q,\\
    &\leq\frac{q}{p}\|v\|_{1,q}^{2q-p}+(\lambda^D_k(q)+\eps)\rho
\end{align*}
for some $v\in P_nA_{\eps,\rho}$ with $\int_{\Omega}|v|^qdx=\rho.$ Since $ P_nA_{\eps}$ is finite-dimensional, there exists a positive constant $D_n(\eps)$ such that
$\int_{\Omega}|\n v|^qdx\leq D_n(\eps) (\int_{\Omega}|v|^qdx)^{p/q}=D_n(\eps)\rho^{q/p}$ and 
$$\|v\|_{1,q}^{2q-p}\leq D_n(\eps)\rho^{\frac{2q-p}{p}}$$
Finally, we get
$$0\leq d_k(p,q;\rho)-\lambda^D_k(q,\rho)\leq D_n(\eps)\rho^{\frac{2q-p}{p}}+\eps\rho\leq (D_n(\eps)+\eps)\rho^{\frac{2q-p}{p}}$$ since $\frac{2q-p}{p}>1.$
\end{proof}
\begin{remark}\label{bifurcation-infty}
    We recall that the $k$-th eigenvalue of equation (\ref{em5}) satisfies 
    $$
    \tilde{\lambda}^D_k(p,q;\rho)\rho =\|w\|^{2(q-p)}_{1,p}\int_{\Omega}|\nabla w|^pdx+\int_{\Omega}|\nabla w|^qdx,~~\text{with}~~\rho=\int_{\Omega}|w|^qdx.
    $$ 
    So, proceding as in Theorem \ref{bifurcation-from-zero} one obtains that $\tilde \lambda^D_k(p,q;\rho)\to {\lambda}^D_k(q)$ as $\rho\to 0^+.$
\end{remark}
\par \medskip
\begin{theorem}\label{binfnon}
The pair $(\lambda^D_k(q),0)$ is a bifurcation point of problem (\ref{em5}) for any $k>0$ and $p<q<\infty.$
\end{theorem}
\begin{proof}
In order to prove Theorem \ref{binfnon}, it suffices to prove that $\tilde{
\lambda}^D_k(p,q;\rho)\rightarrow\lambda^D_k(q,\rho)$ and $\|w_k\|_{1,q}\rightarrow 0$ as $\rho\rightarrow 0^+.$ 
%where $\tilde{\lambda}^D_k(p,q)$ denotes the $k$-th eigenvalue of problem (\ref{em5}) found in Theorem \ref{manysolutions}.
The fact that $\tilde{
\lambda}^D_k(p,q;\rho)\rightarrow\lambda^D_k(q)$ as $\rho\rightarrow 0^+$ follows from Lemma \ref{bifurcainfinity} and Remark \ref{bifurcation-infty}. 

It remains to prove that $\|w_k\|_{1,q}\rightarrow 0$ as $\rho\rightarrow 0^+.$ For any $k>0,$ we have
\begin{equation*}
    \begin{split}
       \|w_k\|_{1,q}^{2(q-p)}\int_{\Omega}|\nabla w_k|^{p}~dx + \int_{\Omega}|\nabla w_k|^{q}dx &=\tilde{\lambda}^D_k(p,q;\rho)\int_{\Omega}|w_k|^q~dx
       \\
       & \le C_k \int_{\Omega}|w_k|^q~dx\\
       & = C_k\, \rho \to 0 \ , \ \hbox{ as } \ \rho \to 0 \end{split}
\end{equation*}
Therefore $\|w_k\|_{1,q} \to 0$, and since $p<q,$ by the H\"older inequality there exists a positive constant $C_1$ such that $\int_{\Omega}|\nabla w_k|^{p}~dx\leq C_1\|w_k\|^p_{1,q}$, and so also $\|w_k\|_{1,p} \to 0$. This completes the proof.
\end{proof}
\par \bigskip
\section{Multiplicity results}\label{mult}
In this section we prove a multiplicity result: we show that for fixed $\lambda \in (\lambda_k^D(q), \lambda_{k+1}^D(q))$ there exist at least $k$ pairs of eigenfunctions \
$\pm u_i^{\lambda}(p,q), i = 1,\dots,k, $ such that $(\lambda,\pm u_i^\lambda(p,q))$ solve equation \eqref{e2}, i.e. 
$$
\lambda = \lambda^D_1(p,q,\rho_1) = \dots = \lambda^D_k(p,q;\rho_k)\ , \ \hbox{ with } \ \rho_i = \int_\Omega |u_i^\lambda(p,q)|^qdx
.$$

We distinguish again the two cases $p<q$ and $p>q$. The proofs rely on variational methods.
\begin{theorem}\label{t3}
Let $1<q<p<\infty$ or $1<p<q<\infty$, and suppose that $\lambda\in(\lambda^D_k(q),\lambda^D_{k+1}(q))$. Then equation (\ref{e1}) has at least $k$ pairs of nontrivial solutions.
\end{theorem}
\begin{proof}
We split the proof into two parts.\\
\\
\textbf{Part 1:}  $p<q.$\\
In this case we will make use of \cite[Proposition 10.8]{AM}. We consider the functional $J_{\lambda}: W^{1,q}_0(\Omega)\backslash\{0\}\rightarrow\mathbb{R}$ associated to the problem (\ref{e1}) defined by
$$J_{\lambda}(u)=\frac{q}{p}\int_{\Omega}|\nabla u|^p~dx+\int_{\Omega}|\nabla u|^q~dx-\lambda\int_{\Omega}|u|^q~dx.$$
The functional $J_{\lambda}$ is not bounded from below on $W^{1,q}_0(\Omega)$, so we consider again the  constraint set $\mathcal{N}_{\lambda}$,  on which we minimize the functional $J_{\lambda}.$ We recall that the constraint set is given by
$$\mathcal{N}_{\lambda}:=\{u\in W^{1,q}_0(\Omega)\backslash\{0\}:~\langle J'_{\lambda}(u),u\rangle=0\}.$$
On $\mathcal{N}_{\lambda},$ we have $J_{\lambda}(u)=(\frac{1}{p}-\frac{1}{q})\displaystyle{\int_{\Omega}}|\nabla u|^p~dx>0.$ We clearly have that
 $J_{\lambda}$ is even and bounded from below on $\mathcal{N}_{\lambda}.$
 Next we show that every Palais-Smale (PS) sequence for $J_{\lambda}$ has a converging subsequence on $\mathcal{N}_{\lambda}.$ Let $(u_n)_{n\geq 0}$ be a (PS) sequence, i.e, $|J_{\lambda}(u_n)|\leq C$, for all $n$, for some $C>0$ and $J'_{\lambda}(u_n)\rightarrow 0$ in $W^{-1,q'}(\Omega)$ as $n\rightarrow +\infty,$ with $\frac{1}{q}+\frac{1}{q'}=1.$  
We first show that the sequence $(u_n)_{n\geq 0}$ is bounded on $\mathcal{N}_{\lambda}$.
Suppose that $(u_n)_{n\geq 0}$ is not bounded, so $\displaystyle{\int_{\Omega}}|\nabla u_n|^q~dx\rightarrow +\infty$ as $n\rightarrow +\infty.$ Since $J_{\lambda}(u_n)=(\frac{1}{p}-\frac{1}{q})\displaystyle{\int}_{\Omega}|\nabla u_n|^p~dx,$ we have $\displaystyle{\int}_{\Omega}|\nabla u_n|^p~dx\leq c.$ On $\mathcal{N}_{\lambda}$, we have
 \begin{equation}\label{ml1}
 0<\int_{\Omega}|\nabla u_n|^p~dx=\lambda\int_{\Omega}|u_n|^q~dx-\int_{\Omega}|\nabla u_n|^q~dx,
 \end{equation}
 and hence $\displaystyle{\int_{\Omega}}|u_n|^q~dx\rightarrow+\infty.$ Let $v_n=\frac{u_n}{\|u_n\|_q}$ then  $\displaystyle{\int_{\Omega}}|\nabla v_n|^q~dx< \lambda$ (using (\ref{ml1})) and hence $v_n$ is bounded in $W^{1,q}_0(\Omega).$ Therefore there exists $v_0\in W^{1,q}_0(\Omega)$ such that $v_n\rightharpoonup v_0$ in $W^{1,q}_0(\Omega)$ and $v_n\rightarrow v_0$ in $L^q(\Omega).$ Dividing (\ref{ml1}) by $\|u_n\|^p_q,$ we have
 $$\frac{\lambda\displaystyle{\int_{\Omega}}|u_n|^q~dx-\int_{\Omega}|\nabla u_n|^q~dx}{\|u_n\|^p_q}=\int_{\Omega}|\nabla v_n|^p~dx\rightarrow 0,$$ since $\lambda\displaystyle{\int_{\Omega}}|u_n|^q~dx-\int_{\Omega}|\nabla u_n|^q~dx=(\frac{1}{p}-\frac{1}{q})^{-1}J_{\lambda}(u_n)$, $|J_{\lambda}(u_n)|\leq C$ and $\|u_n\|^p_q\rightarrow +\infty.$ Now, since $v_n\rightharpoonup v_0$ in $W^{1,q}_0(\Omega)\subset W^{1,p}_0(\Omega),$ we infer that $$\int_{\Omega}|\nabla v_0|^p~dx\leq \liminf_{n\rightarrow +\infty}\int_{\Omega}|\nabla v_n|^p~dx=0,$$ and consequently $v_0=0.$ So $v_n\rightarrow 0$ in $L^q(\Omega)\subset L^p(\Omega)$ and this is a contradiction since $\|v_n\|_q=1.$ Thus $(u_n)_{n\geq 0}$ is bounded on $\mathcal{N}_{\lambda}.$ Now, we show that $u_n$ converges strongly to $u$ in $W^{1,q}_0(\Omega).$\\ We have $\displaystyle{\int_{\Omega}}|u_n|^{q-2}u_n~dx\rightarrow\displaystyle{\int_{\Omega}}|u|^{q-2}u~dx$ as $n\rightarrow\infty$ and since $J_{\lambda}'(u_n)\rightarrow 0$ in $W^{-1,q'}(\Omega),$ $u_n\rightharpoonup u$ in $W^{1,q}_0(\Omega),$ we also have $J_{\lambda}'(u_n)(u_n-u)\rightarrow 0$ and $J_{\lambda}'(u)(u_n-u)\rightarrow 0$ as $n\rightarrow +\infty.$ We recall that with the computations made in Remark \ref{r1}, we have for $1<p<\infty$ 
 $$\int_{\Omega}\left(|\nabla u_n|^{p-2}\nabla u_n-|\nabla u|^{p-2}\nabla u\right)\cdot \nabla(u_n-u)~dx\geq (\|u_n\|^{p-1}_{1,p}-\|u\|^{p-1}_{1,p})(\|u_n\|^{p}_{1,p}-\|u\|^{p}_{1,p})\geq 0
 $$
 Then, \vspace{-0.5cm}
 \begin{eqnarray*}
 \langle J'_{\lambda}(u_n)- J'_{\lambda}(u), u_n-u\rangle &=& q\left[\int_{\Omega}\left(|\nabla u_n|^{p-2}\nabla u_n-|\nabla u|^{p-2}\nabla u\right)\cdot \nabla(u_n-u)~dx\right]\\
 &+& q\left[\int_{\Omega}\left(|\nabla u_n|^{q-2}\nabla u_n-|\nabla u|^{q-2}\nabla u\right)\cdot \nabla(u_n-u)~dx\right]\\ &-& \lambda q\left[\int_{\Omega}\left(| u_n|^{q-2}u_n-| u|^{q-2} u\right)\cdot (u_n-u)~dx\right]\\
 &\geq & q\left[\int_{\Omega}\left(|\nabla u_n|^{q-2}\nabla u_n-|\nabla u|^{q-2}\nabla u\right)\cdot \nabla(u_n-u)~dx\right]\\
 &-& \lambda q\left[\int_{\Omega}\left(| u_n|^{q-2}u_n-| u|^{q-2} u\right)\cdot (u_n-u)~dx\right]. 
 \end{eqnarray*}
 Using Lemma \ref{inevec}, it follows that
 \begin{equation*}
     \begin{split}
         \langle J'_{\lambda}(u_n)- J'_{\lambda}(u), u_n-u\rangle &\geq C\|u_n-u\|^q_{1,q}-\lambda q\left[\int_{\Omega}\left(| u_n|^{q-2}u_n-| u|^{q-2} u\right)\cdot (u_n-u)~dx\right].
     \end{split}
 \end{equation*}
 Therefore $\|u_n-u\|_{1,q}\rightarrow 0$ as $n\rightarrow +\infty$ and $u_n$ converges strongly to $u$ in $W^{1,q}_0(\Omega).$
 \\
 \\
 Let $\Sigma=\{A\subset\mathcal{N}_{\lambda}:~A~\text{closed}~\text{and}~-A=A\}$ and $\Gamma_j=\{A\in\Sigma:~\gamma(A)\geq j\},$ where $\gamma(A)$ denotes the Krasnoselski's genus. We show that $\Gamma_j\neq\emptyset,$ for $j\in \{1,\dots,k\}$.
 \\
 \\
 Let $\lambda\in (\lambda^D_j(q),\lambda^D_{j+1}(q))$ and choose  $S^{\eps}_j\in\Sigma\cap\{\int_{\Omega}|u|^q~dx=1\}$  such that $$\sup_{v\in S^{\eps}_j}\int_{\Omega}|\nabla v|^qdx\leq \lambda^D_j(q)+\eps,~~\eps:=\frac{\lambda-\lambda^D_j(q)}{2}.$$
 Then, for $v\in S^{\eps}_j$
 we set
$$\rho(v)=\left[\frac{\int_{\Omega}|\nabla v|^p~dx}{\lambda\int_{\Omega}|v|^q~dx-\int_{\Omega}|\nabla v|^q~dx}\right]^{\frac{1}{q-p}},$$ with 
\begin{align*}
    \lambda\int_{\Omega}|v|^q~dx-\int_{\Omega}|\nabla v|^q~dx&\geq \lambda\int_{\Omega}|v|^q~dx-(\lambda^D_j(q)+\eps)\int_{\Omega}| v|^q~dx\\
    &=(\lambda-\lambda^D_j(q)-\eps)\int_{\Omega}| v|^q~dx\\
    &=[\lambda-\lambda^D_j(q)-(\frac{\lambda-\lambda^D_j(q)}{2})]\int_{\Omega}| v|^q~dx\\
    &=\frac{\lambda-\lambda^D_j(q)}{2}\int_{\Omega}| v|^q~dx>0,~~\text{for all}~~v\in S^{\eps}_j.
\end{align*} Hence, $\rho(v)v\in \mathcal{N}_{\lambda},$ and then $\rho(S^{\eps}_j)\in\Sigma,$ and $\gamma(\rho(S^{\eps}_j))=\gamma(S^{\eps}_j)=j$ for $1\leq j\leq k.$
\\

It is then standard \cite[Proposition 10.8]{AM} to conclude that $$\sigma_{\lambda,j}=\inf_{A \in \Gamma_j}\sup_{u\in A} J_{\lambda}(u),~~1\leq j\leq k,~~\text{for any }~k\in\mathbb{N}^*$$ yields $k$ pairs of nontrivial critical points for $J_{\lambda},$ which gives rise to $k$ nontrivial solutions of problem (\ref{e1}).
\\
\\
\textbf{Part 2:} $p>q.$\\
In this case, we will rely on the following theorem.
\par \medskip \noindent
\textbf{Theorem} \textup{(Clark, \cite{Cl}) \label{cl1}}.\\
{\it Let $X$ be a Banach space and $G\in C^1(X,\mathbb{R})$ satisfying the Palais-Smale condition with $G(0)=0.$ Let $\Gamma_k =\{~A\in\Sigma~:~\gamma(A)\geq k~\}$ with $\Sigma= \{~A\subset X~;~A=-A~\text{and}~A~\text{closed}~ \}.$ If $c_k=\inf\limits_{A\in \Gamma_k}\sup\limits_{u\in A}G(u)\in (-\infty, 0),$ then $c_k$ is a critical value.}
\par \medskip 
We consider the $C^1$ functional $J_{\lambda}: W^{1,p}_0(\Omega)\subset W^{1,q}_0(\Omega)\rightarrow\mathbb{R}$ $$J_{\lambda}(u)=\frac{q}{p}\int_{\Omega}|\nabla u|^p~dx+\int_{\Omega}|\nabla u|^q~dx-\lambda\int_{\Omega}|u|^q~dx.$$ 
%Let $S^{\eps}_j\in\Sigma\cap\{\int_{\Omega}|u|^q~dx=1\}$ with $\gamma(S^{\eps}_j)=j,$ such that $$\sup_{v\in S^{\eps}_j}\int_{\Omega}|\nabla v|^qdx\leq \lambda^D_j(q)+\eps,~~\eps:=\frac{\lambda-\lambda^D_j(q)}{2}.$$
Let $\Gamma_k=\{A\subset W^{1,q}_0(\Omega)\backslash\{0\},~~A~~\text{compact},~~A=-A,~~\gamma(A)\geq k\},$ and for $\eps>0$ small let $A_{\eps}\in\Gamma_k$ such that $$\sup_{\{u\in A_{\eps},~\int_{\Omega}|u|^qdx=1\}}\int_{\Omega}|\n u|^qdx\leq \lambda^D_k(q)+\eps.$$
We would like to show that 
\begin{equation}\label{ml2}
-\infty<\alpha_{\lambda,k}=\inf_{A\in\Gamma_k}\sup_{u\in A}J_{\lambda}(u)
\end{equation} 
are critical values for $J_{\lambda}.$
We clearly have that $J_{\lambda}(u)$ is  an even functional for all $u\in W^{1,p}_0(\Omega)$, and also $J_{\lambda}$ is bounded from below on $W^{1,p}_0(\Omega)$ since $J_{\lambda}$ is coercive on $W^{1,p}_0(\Omega)$. \\
\\
We show that $J_{\lambda}(u)$ satisfies the (PS) condition. Let $\{u_n\}$ be a Palais-Smale sequence, i.e., $|J_{\lambda}(u_n)|\leq M$ for all $n,$ $M>0$ and $J_{\lambda}'(u_n)\rightarrow 0$ in $W^{-1,p'}(\Omega)$ as $n\rightarrow\infty.$  We first show that $\{u_n\}$ is bounded in $W^{1,p}_0(\Omega).$ We have 
\begin{eqnarray*}
M&\geq & |C\|u_n\|_{1,p}^p-C'\|u_n\|^q_{1,p}|\geq | C\|u_n\|_{1,p}^{p-q}-C'| \|u_n\|_{1,p}^q,
\end{eqnarray*}
and so $\{u_n\}$ is bounded in $W^{1,p}_0(\Omega).$
Therefore, $u\in W^{1,p}_0(\Omega)$ exists such that, up to subsequences that we will denote by $(u_n)_n$ we have $u_n\rightharpoonup u$ in $W^{1,p}_0(\Omega)$ and $u_n\rightarrow u$ in $L^q(\Omega).$ Arguing as in Part 1, we obtain that $\|u_n-u\|_{1,p}\rightarrow 0$ as $n\rightarrow +\infty$, and so $u_n$ converges to $u$ in $W^{1,p}_0(\Omega)\subset W^{1,q}_0(\Omega) .$\\
\par \medskip
As in section \ref{S3}, we approximate $A_{\eps}$ by a finite-dimensional set.
Next, we show that there exists sets $D^{\eps}$ of genus greater of equal to $k$ such that $\sup\limits_{u\in D^{\eps}}J_{\lambda}(u)<0.$
For any $s\in (0,1)$, we define the set $D^{\eps}(s):=s\cdot (P_nA_{\eps})$ and so $\gamma(D^{\eps}(s))=\gamma(P_nA_{\eps})\geq k .$ We have, for any $s\in (0,1)$
\begin{eqnarray}\notag
\sup\limits_{u\in D^{\eps}}J_{\lambda}(u) &= & \sup\limits_{u\in P_nA_{\eps}}J_{\lambda}(su)\\\notag
&\leq& \sup\limits_{u\in P_nA_{\eps}}\left\{\frac{qs^p}{p}\int_{\Omega}|\nabla u|^pdx+s^q\int_{\Omega}|\nabla u|^qdx-\lambda s^q\int_{\Omega}|u|^qdx\right\}\\
&\leq & \sup\limits_{u\in P_nA_{\eps}}\left\{\frac{qs^p}{p}c(n)^p\|u\|^p_{1,q}+s^q(\lambda^D_k(q)+\eps-\lambda)\right\}<0\notag
\end{eqnarray}
for $s>0$ sufficiently small.\\
Finally, we conclude that $\alpha_{\lambda,k}$ are critical values for $J_{\lambda}$ thanks to Clark's Theorem.
\end{proof}
\par \bigskip

The contents of Theorems 2.5-2.6, Theorem 5.3 and Theorem 6.1  are illustrated in the following figure.
\begin{figure}	[!h]\centering
	\includegraphics[width=16cm,height=12cm]{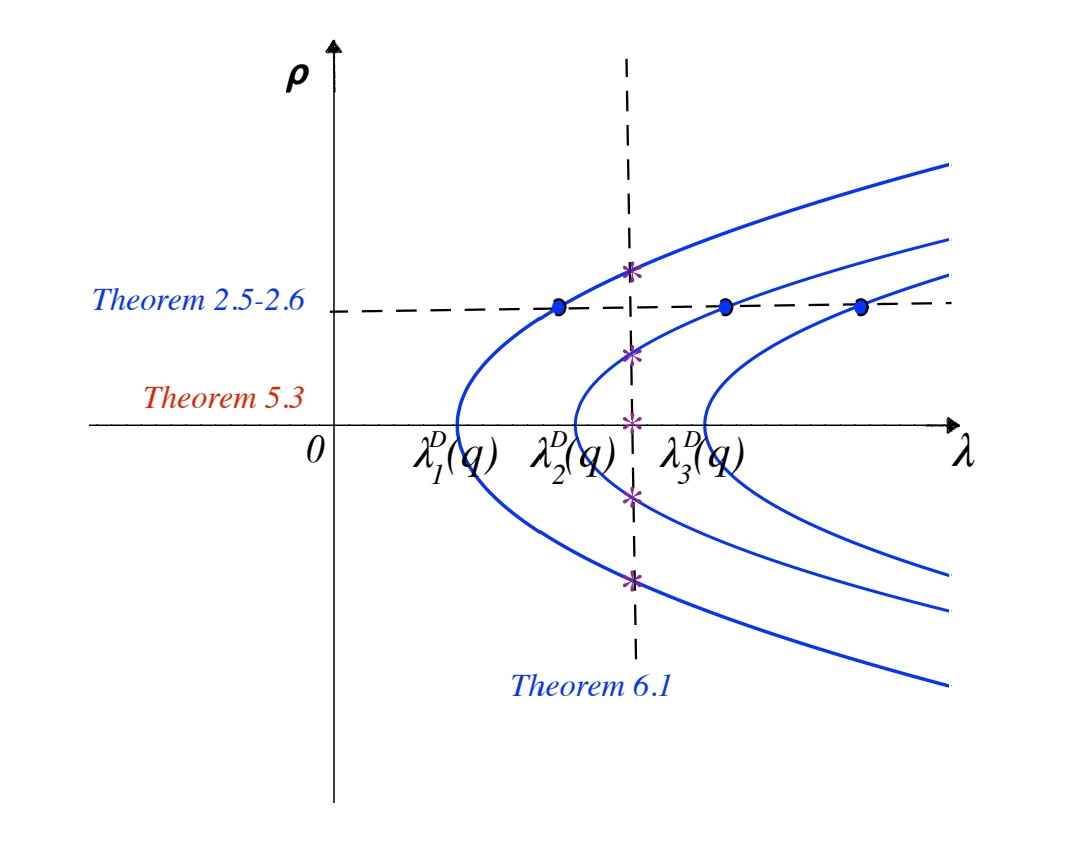}
	\caption{Illustration of the results of Theorem 2.5-2.6, Theorem 5.3 and Theorem 6.1.} \label{fig1}
\end{figure}

\end{document}